\documentclass[12pt]{article}
\usepackage{amsmath,amssymb}
\usepackage{amsthm}
\usepackage{graphicx,tikz}
\usepackage{enumerate}
\usepackage{here}
     \def\ol{\overline}    
   \def\wt{\widetilde}
\def\al{\alpha}               
       \def\eps{\varepsilon}  

\def\om{\omega}              
\def\La{\Lambda}      \def\Si{\Sigma}

\newcommand{\g}[1]{{\mbox{\goth #1}}}
\newcommand{\m}[1]{\mathbb{ #1}}

\newcommand{\gs}[1]{{\mbox{\gots #1}}}
\newfont{\goth}{eufm10 at 12pt}
\newfont{\gots}{eufm8 at 9pt}

\newtheorem{theorem}{Theorem}[section]
\newtheorem{proposition}[theorem]{Proposition}
\newtheorem{conjecture}[theorem]{Conjecture}

\theoremstyle{definition}
\newtheorem{definition}[theorem]{Definition}
\newtheorem{remark}[theorem]{Remark}
\newtheorem{example}[theorem]{Example}
\newtheorem{lemma}[theorem]{Lemma}
\newtheorem{cor}[theorem]{Corollary}
\newtheorem{pro}[theorem]{Proposition}
\numberwithin{equation}{section}
\begin{document}

\title{
Tempered homogeneous spaces III
}

\author{Yves Benoist and Toshiyuki Kobayashi
}
\date{}
\maketitle






\centerline{}

\begin{abstract}{
Let $G$ be a real semisimple algebraic Lie group and $H$ a real 
reductive algebraic subgroup.
We describe the pairs $(G,H)$ for which the  representation  of $G$
in $L^2(G/H)$ is tempered. 

When $G$ and $H$ are complex Lie groups, 
 the temperedness condition is characterized by the fact that
{\it the stabilizer in $H$ of a generic point on $G/H$ is virtually abelian}.
}\end{abstract}

\maketitle

\noindent
\textit{Key words} Lie groups, homogeneous spaces, tempered 
representations, unitary representations, 
matrix coefficients, symmetric spaces

\medskip

\noindent
\textit{MSC (2020):}\enspace
Primary 22E46; 
Secondary 43A85, 
22F30. 

{\footnotesize \tableofcontents}

\section{Introduction}
\label{secintro}


\subsection{Main results}
\label{secmain}

This paper is the third in a series of papers that also include \cite{BeKoI},
\cite{BeKoII} and \cite{BeKoIV}.
In this series we study homogeneous spaces $G/H$
where $G$ is a real semisimple Lie group and $H$ is an algebraic subgroup. 
More precisely, we study the natural
unitary representation of the group $G$ on the Hilbert space $L^2(G/H)$
of square integrable functions on $G/H$. 
In the present paper, as in \cite{BeKoI},
we focus on the case where $H$ is reductive.

We will give a characterization of those homogeneous spaces $G/H$
for which this represen\-tation is tempered. 
We refer to the introduction of both \cite{BeKoI} and \cite{BeKoII}
 for motivations and perspectives on this question.  
In \cite{BeKoI} we discussed the analytic and dynamical part of 
our method.
In this paper we focus on the algebraic part of our method.
Our main result is the following.

\begin{theorem}
\label{thghtemp}
Let  $G$ be a real semisimple algebraic group
and $H$ a real reductive algebraic subgroup.
One has the implications:\\
{\rm{(1)}} If $L^2(G/H)$ is tempered, then the set of points in $G/H$ 
with amenable stabilizer in $H$ is dense.\\
{\rm{(2)}}  If the set of points in $G/H$ 
with virtually abelian stabilizer in $H$ is dense
then
$L^2(G/H)$ is tempered.
\end{theorem}

The proof will also give a complete list of pairs $(G,H)$
 of real reductive algebraic groups
 for which $L^2(G/H)$ 
is not tempered.
When $G$ is simple
 and $H$ is semisimple, this list is given in Tables
\ref{figabcd} and \ref{figefg}
 for the complex case, 
 and in Theorem \ref{thghtempr}
 for the general case.

We recall some terminologies in Theorem \ref{thghtemp}.  
A unitary representation of a locally compact group
 $G$ is said to be {\it{tempered}}
if it is weakly contained in the regular representation in $L^2(G)$,
 see {\it{e.g.,}} \cite[Appendix F]{BeHaVa}. 
An algebraic real Lie group is said to be {\it{amenable}}
if it is a compact extension of a solvable group.  
A group is said to be {\it{virtually abelian}} 
if it contains a finite-index abelian subgroup.

We will see in Section \ref{seccon}
 that in general neither of the converse of these implications
 in Theorem \ref{thghtemp} holds.
However, when $G$ and $H$ are complex Lie groups, 
our implications become an equivalence, since a
reductive amenable complex algebraic
Lie group is always virtually abelian. 

\begin{theorem}
\label{thghtempc}
Let  $G$ be a complex semisimple algebraic group
and $H$ a complex reductive subgroup.
Then the unitary representation of $G$ in 
$L^2(G/H)$ is tempered if and only if the set of points in $G/H$ 
with virtually abelian stabi\-lizer in $H$ is dense.
\end{theorem}

We recall that a semisimple Lie group is said to be {\it{quasisplit}} if 
its minimal parabolic subgroups are solvable.
Then the following result is a particular case of Theorem \ref{thghtempc}.

\begin{example}
\label{exag1ck1c}
Let $G_1$ be a connected real semisimple algebraic Lie group, 
 $K_1$ a maximal compact subgroup,
 and $G_{1,\m C}$ and $K_{1,\m C}$ their complexifications. 
Then the regular representation of $G_{1,\m C}$ in
$L^2(G_{1,\mathbb{C}}/K_{1,\mathbb{C}})$ is tempered if and only if $G_1$ is quasisplit.
\end{example}

%

Theorem \ref{thghtempc} will allow us to give not only a complete description of the pairs $(G,H)$
of {\it complex} reductive algebraic Lie groups
 for which $L^2(G/H)$ is tempered, 
but also a complete description of the pairs $(G,H)$
of {\it real} reductive algebraic Lie groups
 for which $L^2(G/H)$ is tempered. 
The description is as follows. 

Thanks to 
Propositions  \ref{progh1h2} and \ref{progg1grh}, 
we can assume that $G$ is a real simple Lie group and $H$ 
is a real semisimple Lie subgroup of $G$.
The following Theorem \ref{thghtempr} tells us then that
Theorem \ref{thghtempc} is still true for real Lie groups 
except for one list of classical homogeneous spaces 
and three exceptional homogeneous spaces.
We will use Cartan's notation, see \cite[p.518]{He78}, for real  simple Lie algebras.
 
\begin{theorem}
\label{thghtempr}
Let  $G$ be a real simple Lie group
and $H\subset G$ a real semisimple Lie subgroup 
without compact factor, $\g g$ and $\g h$ their Lie algebras.
Then the regular representation of $G$ in 
$L^2(G/H)$ is tempered if and only if one of the following holds: \\
$(i)$ the set of points in $G/H$ 
with virtually abelian stabilizer in $H$ is dense; \\
$(ii)$ $\g g=\g s\g l(2m\!-\!1,\m H)$ and $\g h=\g s\g l(m,\m H)\oplus \g h_2$ 
with $m\geq 1$ and $\g h_2\subset \g s\g l(m\!-\! 1,\m H)$; \\
$(iii)$ $\g g=\g e_{6(-26)}$ and $\g h=\g s\g o(9,1)$ or its subalgebra $\g h=\g s\g o(8,1)$; \\ 
$(iv)$ $\g g=\g e_{6(-14)}$ and 
$\g h=\g s\g o(8,1)$. 
\end{theorem}

In other words, 
 the regular representation of $G$ in $L^2(G/H)$ is tempered 
if and only if either the representation of the complexified Lie group $G_\m C$ in $L^2(G_\m C/H_\m C)$
is tempered or the pair $(\g g,\g h)$ is one of the examples $(ii)$, $(iii)$ or $(iv)$.
We point out that, in examples $(ii)$, $(iii)$ and $(iv)$, the Lie algebra $\g h$
is included in a reductive subalgebra $\widetilde {\mathfrak{h}}$
 such that $({\mathfrak{g}}, \widetilde {\mathfrak{h}})$
 is a symmetric pair.  
We also point out that, in examples $(iii)$ and $(iv)$, 
the real rank of $G$ is $2$ and the real rank of $H$ is $1$.
  
\subsection{What has already been proven in \cite{BeKoI}}
\label{sectrhs}

We need some notations.
Let  $G$ be a real semisimple algebraic group,
$H$ a reductive algebraic subgroup,
$\mathfrak{g}$ and $\mathfrak{h}$ their Lie algebras and $\g q:=\g g/\g h$.
Let $\mathfrak{a}\equiv \mathfrak{a}_\mathfrak{h}$ be a 
maximal split abelian real Lie subalgebra in $\mathfrak{h}$.
Let $V$ be  a finite-dimensional representation of $\g h$:
for instance $V=\g h$ or $V=\g q$
 via the adjoint representation.
Let $Y$ be an element in $\mathfrak{a}$,
we denote by $V_{Y,+}$ the sum of eigenspaces in $V$ of $Y$ having
positive eigenvalues, 
and set
\[
\rho_V(Y)
:= \operatorname{Trace}_{V_{Y,+}} (Y).
\]

According to the temperedness criterion given in \cite[Thm.~4.1]{BeKoI}, 
one has the equiva\-lence
\begin{equation}
\label{eqnl2ghrhrq}
\mbox{$L^2(G/H)$ is tempered}
\Longleftrightarrow
\rho_{\g h}\leq \rho_{\g q}.
\end{equation}
Here the inequality $\rho_{\g h}\leq \rho_{\g q}$ means 
$\rho_{\g h}(Y)\leq \rho_{\g q}(Y)$, 
for all $Y$ in $\g a$.

The generic stabilizers of $G/H$ 
 will be related 
 to those of ${\mathfrak{q}}={\mathfrak{g}}/{\mathfrak{h}}$
 in Section \ref{secgenstaqgh}. 
Thus Theorem \ref{thghtemp} is a consequence of the following 
Theorem \ref{thrhrq}. 

\begin{theorem}
\label{thrhrq}
Let  $\g h\subset \g g$ be a pair of real semisimple Lie algebras
and $\g q=\g g/\g h$. One has the implications:\\
{\rm{(1)}} $\rho_{\g h}\leq \rho_{\g q}
\Longrightarrow $ the set of points in $\g q$ 
with amenable stabilizer in $H$ is dense;\\
{\rm{(2)}}  the set of points in $\g q$ 
with abelian stabilizer in $\g h$ is dense
$\Longrightarrow \rho_{\g h}\leq \rho_{\g q}$.
\end{theorem}

The first  implication of Theorem \ref{thrhrq}
 is proven in Proposition \ref{prorhrv}
 by a short argument based on a {\it{slice theorem}} near a generic orbit.
The proof of the converse implication is much longer.
We will reduce it in Lemma \ref{lemrhrvc} to the case where $\g g$ and $\g h$ 
are complex and semisimple Lie algebras, 
{\it{i.e.}} we will have to prove
the following Theorem \ref{thrhrqc} which is a special case of Theorem \ref{thrhrq}.

\begin{theorem}
\label{thrhrqc}
Let  $\g h\subset \g g$ be two complex semisimple Lie algebras
and $\g q=\g g/\g h$. One has the  equivalence:\\
 $\rho_{\g h}\leq \rho_{\g q}$ $\Longleftrightarrow$ the set of points in $\g q$ 
with abelian stabilizer in $\g h$ is dense.
\end{theorem}

Similarly, we can deduce Theorem \ref{thghtempr} from the following 
Theorem \ref{thrhrqr}
 by the criterion \eqref{eqnl2ghrhrq}.

\begin{theorem}
\label{thrhrqr}
Let  $\g g$ be a real simple Lie algebra and $\g h$ 
a semisimple Lie subalgebra such that the adjoint group of $\g h$ has no 
compact factor. Then 
\\
 $\rho_\g h\leq \rho_\g q$
 if and only if one of the following holds:
\\
$(i)$ the set of points in $\g q$ 
with abelian stabilizer in $\g h$ is dense;\\
$(ii)$ $\g g=\g s\g l(2m\!-\!1,\m H)$ and $\g h=\g s\g l(m,\m H)\oplus \g h_2$ 
with $m\geq 1$ and $\g h_2\subset \g s\g l(m\!-\! 1,\m H)$;\\
$(iii)$ $\g g=\g e_{6(-26)}$ and $\g h=\g s\g o(9,1)$ or its subalgebra $\g h=\g s\g o(8,1)$;\\ 
$(iv)$ $\g g=\g e_{6(-14)}$ and 
$\g h=\g s\g o(8,1)$.  
\end{theorem}

\subsection{Strategy of proof of Theorem \ref{thrhrqc}
 for complex ${\mathfrak{g}}$}
\label{secstrac}

As we have already mentioned, 
 we give a short proof for the implication 
 $\Rightarrow$ of Theorem \ref{thrhrqc} 
 in Chapter \ref{subsecnot}, 
 and only the converse implication remains to be proven.
We reduce our analysis in Proposition \ref{prorhrqreacom} 
to the case where $\g g$ is simple. 
When $\g g$ is simple we give a complete classification of the semisimple 
Lie subalgebras $\g h\subset \g g$ for which $\rho_{\g h}\not\leq \rho_{\g q}$
 in Tables \ref{figabcd} and \ref{figefg}
and we compute in each case the generic stabilizer.
The proof lasts from Chapter \ref{seccla} to Chapter \ref{secbpvnon}.

When $\g g$ is simple and classical 
{\it{i.e.}} $\g g=\g s\g l(n,\m C)$, $\g g=\g s\g o(n,\m C)$
or $\g g=\g s\g p(n,\m C)$,
 the list of such pairs $(\g g, \g h)$
 is given in Table \ref{figabcd}
 in Chapter \ref{subsecnot}. 
In order to check this list, 
we first deal with the case 
when the standard representation of $\g g$
 ({\it{i.e.}} $\m C^n$ for ${\mathfrak{s l}}(n,\m C)$ and ${\mathfrak{s o}}(n,\m C)$, 
 and $\m C^{2n}$ for ${\mathfrak{s p}}(n,\m C)$)
 remains irreducible 
 as a representation of the subalgebra $\g h$ in Section \ref{secclairr}, 
then we deal with the case where it is reducible
in Section  \ref{secclared1}.

When $\g g$ is exceptional,
 the list of such pairs $(\g g, \g h)$
is given in Table \ref{figefg} in Chapter \ref{seccla}.
In order to check this list, 
we use Dynkin's list (Tables \ref{figral} and \ref{figsal}) 
in \cite{Dyn52} of maximal 
semisimple Lie subalgebras $\g h$ in $\g g$ (up to conjugacy).
We extract, in Section \ref{secpromax}, from Dynkin's classification
 those $\g h$ 
for which $\rho_{\g h}\not\leq \rho_{\g q}$. 
Then using this first list,
we give, in Section \ref{secpronon}, 
the list of the semisimple Lie algebras $\g h$ 
with $\rho_{\g h}\not\leq \rho_{\g q}$ which are maximal in one of 
the Lie algebras of the first list. We prove then that there are no 
other possibilities for $\g h$ (Lemma \ref{lemmamama}).  

All this analysis relies on explicit upper bounds for an invariant 
$p_V$ associated to any finite-dimensional representation $V$
of $\g h$ (see Equation \eqref{eqnpvinf}).
The proof of these upper bounds are given in Chapter \ref{secbpvsim} 
when $\g h$ is simple and in Chapter \ref{secbpvnon} 
when $\g h$ is not simple.

\subsection{Strategy of the proof of Theorem \ref{thrhrqr} for real ${\mathfrak{g}}$}
\label{secstrar}

The proof occupies Chapter \ref{secrearedsub}.  
The implication $(i)\Rightarrow$ $\rho_{\g h} \le \rho_{\g q}$
 reduces to the complex case
 (Theorem \ref{thrhrqc})
 by Proposition \ref{proortredsta}
 and Lemma \ref{lemrhrvc}, 
 whereas the implication $(ii)$, $(iii)$ or $(iv)$
 $\Rightarrow$ $\rho_{\g h} \le \rho_{\g q}$
 is straightforward.  
To see the converse implication, 
 let $\g g$ be a real simple Lie algebra and $\g h\subsetneq \g g$
be a real semisimple Lie subalgebra 
 for which the group $\operatorname{Aut} (\g h)$
 of automorphisms has no compact factor.
We assume that this pair $(\g g,\g h)$ does not satisfy $(i)$, 
 or equivalently, 
 $\rho_{\g h_\m C}\not\leq\rho_{\g q_\m C}$ by Theorem \ref{thrhrqr}, 
 and we want to check that, except for cases $(ii)$, $(iii)$ and $(iv)$,
one has also $\rho_{\g h}\not\leq\rho_{\g q}$.

When the complexified Lie algebra $\g g_\m C$ is not simple, 
 equivalently when the Lie algebra $\g g$ has a complex structure, 
 we prove in Proposition \ref{progchr}
that $\g h$ contains a complex semisimple Lie subalgebra $\g h_0$
such that the pair $(\g g_0, \g h_0)$ already satisfies 
$
\rho_{\g h_0}\not\leq\rho_{\g g/\g h_0}
$.

When $\g g_{\m C}$ is  simple, we know that the pair 
$(\g g_{\m C},\g h_{\m C})$ is in Tables \ref{figabcd} or \ref{figefg}.
For each pair in these tables, we know that there exists a {\it{witness vector}}
$X$ in the Cartan subalgebra $\g j_{\m C}$ of $\g h_{\m C}$, 
{\it{i.e.}} an element $X$ such that 
$\rho_{\g h_\m C}(X)>\rho_{\g q_\m C}(X)$
 (Definition \ref{def:witness}).
The main point is to find this witness $X$ in a maximal split abelian subalgebra $\g a_{\g h}$ of $\g h$. 
Using the Satake diagram of $\g h$ which describes the embedding 
$\g a_{\g h}\subset \g j_{\m C} \cap \g h$,
we will check,
 based on Tables \ref{figabcdr} and \ref{figefgr}, 
that it is always  possible to find a witness vector $X$
 in $\g a_{\g h}$, except 
in  cases $(ii)$, $(iii)$ and $(iv)$.

\subsection{Comments on the proof}
\label{seccompro}

This text was written more than five years ago. 
Indeed we believe that there should exist a shorter proof 
 for the implication $\Leftarrow$ in Theorem \ref{thrhrqc} 
which does not rely on a case-by-case analysis.
This is why we delayed its publication  trying to find such a simpler proof.
This is also why we present in this text only the main structure
and ideas of our long proof 
leaving the lengthy calculations to the reader.

Relying on Theorem \ref{thghtempc}, we 
found recently in \cite{BeKoIV} various temperedness criteria for $L^2(G/H)$
valid for complex algebraic subgroups $H$ of  complex semisimple 
Lie groups $G$.

\vskip 1pc
{\bf{Acknowledgments.}}\enspace
The authors 
are grateful to the IHES
 and to The University of Tokyo for their support.  
The second author was partially supported
 by JSPS Kakenhi Grant Number JP18H03669.  
\section{Notations and preliminary reductions}
\label{subsecnot}

In this chapter,
 we prove the first assertion 
 of Theorem \ref{thrhrq}, 
 and explain how the second assertion of Theorem \ref{thrhrq}
 can be deduced from Theorem \ref{thrhrqc}. 
Then the proof of Theorem \ref{thrhrqc} is reduced
 to the case where $\g g$ is simple, 
 for which we shall discuss in Chapters \ref{seccla}--\ref{secbpvnon}. 

\subsection{Reductive generic stabilizer}
\label{secrgs}

Let $\g h$ be a  real semisimple Lie algebra,
and $V$ a finite-dimensional represen\-tation of $\g h$ over $\m R$. 
For $v$ in $V$,
 we denote by $\g h_v\equiv \operatorname{Stab}_{\g h}(v)$
 the stabilizer of $v$ in $\g h$.
We recall that $\g h_v$ is said to be {\it{reductive}}
 if the adjoint representation of ${\mathfrak{h}}_v$ on ${\mathfrak{h}}$
 is semisimple, 
 or equivalently, 
 if the action of $\g h_v$ on $V$ is semisimple.
  
\begin{definition}
[RGS]
\label{defrgs}
We say that $V$ has RGS in $\g h$ if the set 
$\{ v\in V\mid \g h_v
\;\mbox{is reductive}\}$ is dense in $V$.
\end{definition}

Here, ``dense'' means dense for the locally compact topology.
We can equivalently replace in this definitions dense by Zariski dense. 

We say that the representation of $\g h$ in $V$ is {\it{self-dual}} 
if it is equivalent to the contragredient representation
 in the dual space $V^*$.
We say that the representation of $\g h$ in $V$ is {\it{orthogonal}}
(resp. {\it{symplectic}}) if it preserves a nondegenerate symmetric 
(resp. skew-symmetric) bilinear form on $V$. 

For instance, when $\g h$ is a semisimple Lie subalgebra of a 
semisimple Lie algebra $\g g$, then the representation 
of $\g h$ in $\g g/\g h$ is orthogonal. Indeed since the restriction 
to $\g h$ of the Killing form of $\g g$ is nondegenerate, one can identify
$\g g/\g h$ with the orthogonal complementary subspace of $\g h$ in $\g g$.
On this space the action of $\g h$ preserves the restriction
 of the Killing form.

\begin{proposition}
\label{proortredsta}
Let $\g h$ be a real semisimple Lie algebra, 
and $V$ an orthogonal  finite-dimensional representation of $\g h$. 
Then $V$ has RGS in $\g h$.
\end{proposition}

This will follow from the  general lemmas below.

We denote by $\g h_\m C$ the complexified Lie algebra of ${\mathfrak{h}}$, 
 and by $V_\m C$
the complexified representation of $\g h_\m C$.
We say that the representation of $\g h_\m C$ in $V_\m C$ is orthogonal
(resp. symplectic) if it preserves a nondegenerate symmetric 
(resp. skew-symmetric) complex bilinear form. 


\begin{lemma}
\label{lemortcom}
Let $\g g$ be a complex semisimple Lie subalgebra of $\g s\g o(n,\m C)$.
Then there exists a real Lie subalgebra $\g h$ of $\g s\g o(n,\m R)$
such that its complexification $\g h_\m C$ is ${\rm SO}(n,\m C)$-conjugate to $\g g$.
\end{lemma}

\begin{proof}
Let $G$ be the connected algebraic subgroup of ${\rm SO}(n,\m C)$
with Lie algebra $\g g$, $K$ a maximal compact subgroup of
$G$ and $\g k$ its Lie algebra. 
Then we can find $g \in {\rm SO}(n,\m C)$
 such that $H:=g K g^{-1}$ is contained in the maximal compact subgroup 
${\rm SO}(n,\m R)$ of ${\rm SO}(n,\m C)$. 
Since $\g g=\g k_\m C$, we are done.  
\end{proof}


\begin{lemma}
\label{lemredcom}
Let $\g g$ be a complex semisimple Lie algebra 
and $W$ a finite-dimensional representation of $\g g$. 
Assume that $W$ has RGS in $\g g$.
Then there exists a reductive Lie subalgebra $\g m$ of $\g g$ such that the set 
of $w$ in $W$ whose sta\-bilizer $\g g_w$ is conjugate to $\g m$
contains a non-empty Zariski open subset of $W$.
\end{lemma}

``Conjugate'' means a ``conjugate by the adjoint group $G$ of $\g g$''.
We say 
 that the Lie algebra $\g m$ in Lemma \ref{lemredcom} is the {\it{generic stabilizer}} of $V$.  
It is well defined only up to conjugacy.

\begin{proof}[Proof of Lemma \ref{lemredcom}]
For $w$ in $W$, we denote by $\g u_w$ the unipotent radical 
of its stabilizer $\g g_w$ in $\g g$, that is  $\g u_w$ is the largest 
nilpotent ideal of  $\g g_w$ all of whose elements are nilpotent.
Let $\displaystyle d:=\min_{w\in W}\dim \g g_w$.
We introduce the Zariski open subsets,
$W':=\{w\!\in\! W\mid \dim \g g_w =d\}$ and 
$W'':=\{w\!\in\! W'\mid \dim\g u_w =0\}$.
By assumption the set $W''$ is a non-empty Zariski open set. 
In particular it is connected.
Since the set of conjugacy classes of reductive algebraic Lie subalgebras of $\g g$ is countable,
the map $w \mapsto \g g_w$ must be constant modulo conjugation on $W''$.  
\end{proof}


\begin{lemma}
\label{lemredstacom}
Let $\g h$ be a real semisimple Lie algebra, 
 and $V$ a finite-dimen\-sional representation of $\g h$ over ${\mathbb{R}}$. 
One has the equivalences: \\
{\rm{(1)}} $V$ is  orthogonal 
$\Longleftrightarrow$
$V_{\m C}$ is  orthogonal.\\
{\rm{(2)}} $V$ has RGS in $\g h$ $\Longleftrightarrow$
$V_\m C$ has RGS in $\g h_{\m C}$.
\end{lemma}
\begin{proof}
(1) The implication $\Rightarrow$ is obvious.  
Conversely, 
suppose that the representation of $\g h_{\m C}$ in $V_{\m C}$
is orthogonal.  
Then one has two $\g h$-invariant symmetric bilinear forms
$A, B \colon V \times V \to {\mathbb{R}}$
such that $A + \sqrt{-1} B$ is nondegenerate.  
In turn, one can find $t \in {\mathbb{R}}$
such that $A+ t B$ is nondegenerate, 
showing that the representation of $\g h$ in $V$ is orthogonal.  
 
(2) 
As in the proof of Lemma \ref{lemredcom}, 
for $v$ in $V_\m C$, we denote by $\g u_{\m C,v}$ the unipotent radical of its stabilizer $\g h_{\m C,v}$ in $\g h_\m C$.

Let $\displaystyle d:=\min_{v\in V_\m C}\dim \g h_{\m C,v}$ and
$V_{\m C}':=\{v\!\in\! V_{\m C}\mid \dim \g h_{\m C,v} =d\}$. 
Since $V_{\m C}'$ is Zariski open, it meets $V$
and one has  $\displaystyle d=\min_{v\in V}\dim \g h_{v}$. 

One then introduces 
$\displaystyle \delta :=\min_{v\in V'_\m C}\dim \g u_{\m C,v}$ and
$V_{\m C}'':=\{v\!\in\! V_{\m C}'\mid \dim\g u_{\m C,v} =\delta\}$.
Since $V_{\m C}''$ is Zariski open, it meets $V$
and one has the equivalences:\\
$V$ has RGS in $\g h$ $\Longleftrightarrow$
$\delta=0$
$\Longleftrightarrow$
$V_\m C$ has RGS in $\g h_{\m C}$.
\end{proof}
In Lemma \ref{lemredstacom}
{\rm{(2)}},
 there exist finitely many reductive Lie subalgebras
$\g m_1,\ldots, \g m_r$ of $\g h$
such that the set of $w$ in $V$
 whose stabilizer $\g h_w$ is conjugate 
to one of the $\g m_i$
contains a non-empty Zariski open subset of $V$.  

``Conjugate'' means a  ``conjugate by the adjoint group $H$ of $\g h$''.
We say that the Lie algebras $\g m_i$ which cannot be removed from this list in Lemma \ref{lemredstacom} 
are the {\it{generic stabilizers}} of $V$.
They  are well-defined only up to conjugacy and permutation.



\begin{proof}[Proof of Proposition \ref{proortredsta}]
We extend the quadratic form on $V$ to a complex quadratic form on $V_\m C$.
Applying Lemma \ref{lemortcom} to the complexified representation of 
$\g h_{\m C}$ in $V_\m C$, one finds
a real form $\g k$ of the Lie algebra $\g h_{\m C}$
and a $\g k$-invariant real form $V_0$ of $V_\m C$
such that the restriction of the quadratic form to $V_0$ is positive.
Since all the Lie subalgebras of $\g s\g o(n,\m R)$ are reductive, 
$V_0$ has RGS in $\g k$. 
Applying twice Lemma \ref{lemredstacom} (2), we deduce 
successively that $V_\m C$
has RGS in $\g h_\m C$ and  that $V$ has RGS in $\g h$.
\end{proof}

\subsection{Function $\rho_V$ and invariant $p_V$ }
\label{secrhoV}

Let $\g h$ be a real Lie algebra, 
 and $V$ a finite-dimensional representation of $\g h$ over ${\mathbb{R}}$.  
For an element $Y$ in $\g h$, 
 we consider eigenvalues
 of $Y$ in the complexification $V_{\mathbb{C}}$, 
 and write $V_{\mathbb{C}}=V_+ \oplus V_0 \oplus V_-$
 for the direct sum decomposition into the largest vector subspaces
 of $V_{\mathbb{C}}$
 on which the real part of all the (generalized) eigenvalues
 of $Y$ are positive, zero, and negative, respectively.  
We define the non-negative functions  
 $\rho_V^+$ and $\rho_V$ on $\g h$ 
 by 
\begin{align*}
  \rho_V^+(Y) :=& \text{the real part of }\operatorname{Trace} (Y|_{V_+}), 
\\
  \rho_V(Y):=&\frac 1 2 (\rho_V^+(Y) + \rho_V^+(-Y)), 
\end{align*}
where $\operatorname{Trace}$ denotes the trace
 of an endomorphism of a vector space.

By definition,
 one always has the equality
 $\rho_V(-Y)=\rho_V(Y)$.  
Moreover, 
 when the action of $\g h$ on $V$
 is trace-free, 
 one has the equality
\[
  \text{$\rho_V(Y)=\rho_V^+(Y)$
  \qquad
  for all $Y \in \g h$.  }
\]
The function denoted by $\rho_V$
 in \cite[Sect.~3.1]{BeKoI}
 is what we call now $\rho_V^+$.

Suppose $\g h$ is a real reductive Lie algebra and $V$ is a semisimple representation.
Let $\mathfrak{a}\equiv \mathfrak{a}_\mathfrak{h}$ be a 
maximal split abelian real Lie subalgebra in $\mathfrak{h}$.
This subalgebra is a real vector space whose dimension $\ell$
is the real rank of $\g h$, 
 to be denoted by $\operatorname{rank}_{\mathbb{R}}\g h$. 
Then $\rho_V$ is determined completely
 by its restriction to $\g a$, 
 and actual computations of $\rho_V$
 in Chapters \ref{secbpvsim}--\ref{secrearedsub}
 will be carried out 
 by using the weight decomposition of $V$
 with respect to $\g a$, 
 which we explain now.  
Let $P(V,\mathfrak{a})$ be the set of weights of $\g a$ is $V$,
{\it{i.e.}} the set of linear forms $\al\in \g a^*$ for which the weight space 
$V_\al:=\{ v\in V\mid Yv=\al(Y)v\;, \; \forall \,Y\in \g a\}$ is nonzero.
For such a weight $\al$ we set  $m_\al:=\dim V_\al$ 
and $|\al |:=\max(\al,-\al)$.

By definition the restriction of $\rho_V$ to the subspace $\g a$
 is given by the formula 
\begin{equation}\label{eqnrhoV}
\textstyle
\rho_V
= \tfrac{1}{2}\sum m_\al |\al |\; , 
\end{equation}
where the sum is taken over all the weights $\al\in P(V,\mathfrak{a})$.  

Since this function $\rho_V \colon \mathfrak{a} \to \mathbb{R}_{\ge0}$ 
is very important in our analysis,
we begin with a few elementary but useful comments. 
This function $\rho_V$ is  invariant under the
Weyl group 
$W$ of the (restricted) root system
$\Sigma(\mathfrak{h},\mathfrak{a})$.
Moreover the function $\rho_V$ is convex,
continuous and is piecewise linear
in the sense 
 that there exist finitely many convex polyhedral cones 
which cover $\mathfrak{a}$ and on which $\rho_V$ is linear.

For two real semisimple representations $V'$, $V''$ of ${\mathfrak{h}}$, 
 one has
\begin{equation}
\label{eqn:rhosum}
\rho_{V' \oplus V''}=\rho_{V'}+\rho_{V''}.   
\end{equation}
We denote by $V^{\ast}$ the contragredient representation of $V$.  
Then one has
\begin{align}
\label{eqn:rhodual}
\rho_{V^{\ast}}&=\rho_{V},    
\\
\label{eqn:rhoVVstar}
\rho_{V \oplus V^{\ast}}&=2\rho_{V}.   
\end{align}
When $V$ is self-dual, 
{\it{i.e.}}, 
 when $V^{\ast}$ is isomorphic to $V$ 
 as an ${\mathfrak{h}}$-module, 
each nonzero weight $\al$ occurs in pair with its opposite
 $-\al$
 and $\rho_V$ is equal to 
\[
   \rho_V^+
= \sum_{\al\in P(V,\g a)} m_\al \al_+
\]
 where $\al_+:=\max(\al,0)$.

When $V=\g h$ is the adjoint representation, 
this function $\rho_{\mathfrak{h}}$ coincides
 with twice the \lq\lq{usual $\rho$}\rq\rq\
 on a positive Weyl chamber $\mathfrak{a}_+$ with respect to the positive
system $\Sigma^+(\mathfrak{h},\mathfrak{a})$.
For other representations $V$, the maximal convex polyhedral cones on which 
$\rho_V$ is linear are most often much smaller than the Weyl chambers.

We introduce the invariant of an ${\mathfrak{h}}$-module $V$ by
\begin{equation}
\label{eqnpvinf}
p_V:=\inf\{ t>0\mid  
\rho_\g h\leq t\,\rho_V\}.  
\end{equation}
By definition,
 $p_V=\infty$ if $V$ has nonzero fixed vectors
 of $\g h$. 
In general, 
 for a finite-dimensional representation of $\mathfrak h$
 on a real vector space $V$, 
 one has the equivalences:
\begin{equation}
\label{eqnpv1}
\rho_\g h\leq \rho_V
\Longleftrightarrow
p_V\leq 1\, ,
\end{equation}
\begin{equation}
\label{eqnpv2}
\rho_\g h\leq \rho_{V\oplus V^*}
\Longleftrightarrow
p_V\leq 2\, .
\end{equation}
Let us explain why this invariant $p_V$ is relevant.  
Indeed, the main results of \cite{BeKoI} may be reformulated as follows.  
We recall that a unitary representation $\pi$ 
 of a locally compact group $G$
 on a Hilbert space ${\mathcal{H}}$ is called
 {\it{almost}} $L^p$ ($p \ge 2$)
 if there exists a dense subset $D \subset {\mathcal{H}}$
 for which the matrix coefficients $g \mapsto (\pi(g)u,v)$
 are in $L^{p+\varepsilon}(G)$
 for all $\varepsilon>0$
 and all $u,v$ in $D$.  
If $G$ is a semisimple Lie group, 
 $\pi$ is almost $L^2$ 
 if and only if $\pi$ is tempered
 \cite{CHH}.  
Suppose $H$ is a real reductive algebraic Lie group
 with Lie algebra $\g h$.  
For a positive even integer $p$ 
 and for an algebraic representation 
 $H \to S L(V)$, 
 one has the following equivalences 
 (\cite[Thm.~3.2]{BeKoI})
\begin{align*}
\text{$L^2(V)$ is $H$-tempered}
&\Longleftrightarrow p_V \le 2, 
\\
\text{$L^2(V)$ is $H$-almost $L^p$}&\Longleftrightarrow p_V \le p.  
\end{align*}
Moreover, 
 for a real semisimple algebraic group $G$ 
and a real reductive algebraic subgroup $H$, 
 one has the following equivalences
 (\cite[Thm.~4.1]{BeKoI}):
\begin{align}
\text{$L^2(G/H)$ is $G$-tempered}
&\Longleftrightarrow p_{\mathfrak{g}/\mathfrak{h}} \le 1, 
\label{eqn:ghpone}
\\
\text{$L^2(G/H)$ is $G$-almost $L^p$}&\Longleftrightarrow p_{\mathfrak{g}/\mathfrak{h}} \le p-1.  
\notag
\end{align}
The inequality $\rho_{\mathfrak{g}/\mathfrak{h}} \le 1$
 in \eqref{eqn:ghpone} is nothing but the criterion \eqref{eqnl2ghrhrq}
 by \eqref{eqnpv1}.

Hence, 
we would like to describe all the
ortho\-gonal representations $V$
 such that $p_V \le 1$, 
 {\it{i.e.,}} $\rho_\g h\leq \rho_V$.
In particular, we would like to describe all the representations $V$ 
 such that
 $p_V \le 2$, 
 or equivalently, 
  $\rho_\g h\leq \rho_{V\oplus V^*}$.

We end this section by a useful remark.
We note that, when $V$ is a direct sum
 of two subrepresentations $V=V'\oplus V''$,
then one has the inequality
\begin{equation}
\label{eqnpvpvpv}
p_V^{-1}\geq p_{V'}^{-1}+p_{V''}^{-1}
\end{equation}
as one sees from \eqref{eqn:rhosum} and from the following equivalent definition of $p_V$:
\begin{equation}
\label{eqn:pvmin}
 p_V^{-1}=\min_{Y \in \g a \setminus \{0\}} \frac{\rho_V(Y)}{\rho_{\g h}(Y)}.  
\end{equation}
In general, 
 the equality in \eqref{eqnpvpvpv} may not hold, 
 but if $V$ is of the form
\[
   V= \underbrace{V' \oplus \cdots \oplus V'}_m
\oplus \underbrace{(V')^{\ast} \oplus \cdots \oplus (V')^{\ast}}_n, 
\]
then one has
\begin{equation}
\label{eqn:pvmult}
p_V=\frac{1}{m+n} p_{V'}.  
\end{equation}

\subsection{Abelian and amenable generic stabilizer}
\label{secagsamgs}

Let $\g h$ be a  real reductive Lie algebra.  
We say that a subalgebra $\g l$ is {\it{amenable reductive}} if it is reductive and
if the restriction
of the Killing form of $\g h$ to $[\g l, \g l]$ is negative.
Let $V$ be a  finite-dimensional represen\-tation of $\g h$, 
 and $\g h_v:=\{X \in \g h | X \cdot v=0\}$
 for $v \in V$.

\begin{definition}
[AGS and AmGS]
\label{defags}
We say that\\
$V$ has {\rm{AGS}} in $\g h$ if the set 
$\{ v\in V\mid \g h_v
\;\mbox{is abelian reductive}\}$ is dense in $V$.\\
$V$ has AmGS in $\g h$ if the set 
$\{ v\!\in\! V\mid \g h_v\;
\mbox{is amenable reductive}\}$ is dense in $V$.
\end{definition}

\begin{remark}
In the first definition, it is equivalent to say 
Zariski dense instead of dense. However in the second definition,
it is not equivalent to say Zariski dense instead of dense.
Indeed, in the natural representation $\m R^4$ of $\g s\g o(3,1)$, the set of 
points $v$ with reductive amenable stabilizer is Zariski dense but is not dense.\end{remark}

The statement (1) in Theorem \ref{thrhrq} is a special case of the following proposition.

\begin{proposition}
\label{prorhrv}
Let  $\g h$ be a real semisimple Lie algebra
and $V$ an orthogonal representation of $\g h$. One has the implication:
\begin{eqnarray}
\rho_{\g h}\leq \rho_{V}
&\Longrightarrow &
\mbox{$V$ has {\rm{AmGS}} in $\g h$.}
\end{eqnarray}
Moreover,
 if one of the generic stabilizers $\g m$ of $V$ has the same real rank as $\g h$,
then  the converse is true. 
\end{proposition}

\begin{proof}
By Proposition \ref{proortredsta}, 
 ${\mathfrak{m}}$ is reductive.  
Then Proposition \ref{prorhrv} follows from Lemma \ref{lemrhrv} below and from the equivalence
for a reductive Lie algebra $\g m$~: 
$\rho_\g m=0\Longleftrightarrow \g m$ is amenable.
\end{proof}

\begin{lemma}
\label{lemrhrv}
Let  $\g h$ be a real semisimple Lie algebra,
$V$ an orthogonal repre\-sen\-tation of $\g h$, 
 and $\g m$ one among 
the finitely many generic stabilizers $\g m_i$ of $V$.
Let $t\geq 1$. One has the implication:
\begin{eqnarray}
\rho_{\g h}\leq t\rho_{V}
&\Longrightarrow &
\rho_{\g m}\leq (t-1)\rho_{\g h/\g m}.
\end{eqnarray}
Moreover, if 
$\g m$ and $\g h$ have the same real rank, the converse is true. 
\end{lemma}

Let $\g a_\g m\subset\g a_\g h$ be Cartan subspaces of $\g m$ and $\g h$.
We recall that $\rho_{\g h}$ and $\rho_{V}$
are functions on $\g a_\g h$ while $\rho_{\g m}$
and $\rho_{\g h/\g m}$ are functions on $\g a_\g m$.

\begin{proof} We can assume the representation of $\g h$ to be faithful.
Let $H$ be an algebraic subgroup of ${\rm GL}(V)$ with Lie algebra $\g h$.  
Since $V$ has RGS in $\g h$, we can find a slice $\Si$ of points $v$ of $V$ 
whose stabilizer $M$ in $H$ has Lie algebra $\g m$ and 
an open neighborhood of $v$ foliated by $H$-orbits.
The tangent space at $v$ to the orbit $Hv$ is isomorphic 
as a representation of  $M$ to $\g h/\g m$.
Since $M$ preserves the leaves of the foliation 
the quotient $V/(\g h/\g m)$ is a trivial representation of $\g m$.
Hence, for $X$  in $\g a_\g m$, one has the equivalences~:
\begin{eqnarray*}
\rho_\g h(X)\leq t\,\rho_V(X)&\Longleftrightarrow&
\rho_\g h(X)\leq t\,\rho_{\g h/\g m}(X)\\
&\Longleftrightarrow&
\rho_\g m(X)\leq (t-1)\,\rho_{\g h/\g m}(X).
\end{eqnarray*}
Our claims follow since, if $\g h$ and $\g m$ have the same real rank,
one has $\g a_\g m=\g a_\g h$.
\end{proof}

The converse to Proposition \ref{prorhrv} is not true,
but we conjecture that a kind of converse is true:

\begin{conjecture}
\label{conrhrv}
Let  $\g h$ be a real semisimple Lie algebra
and $V$ an orthogonal representation of $\g h$. One has the implications:
\begin{eqnarray}
\mbox{$V$ has $\rm{AGS}$ in $\g h$}
&\Longrightarrow &
\rho_{\g h} \leq \rho_{V}.
\end{eqnarray}
\end{conjecture}

\begin{remark}
\label{rem:conrhrv}
We shall see that Conjecture \ref{conrhrv} holds
 in the following settings:
\\
{\rm{(1)}}\enspace
$\g h$ is simple
 (Corollary \ref{corirrsim});
\\
{\rm{(2)}}\enspace
there is a semisimple Lie algebra $\g g$ containing $\g h$
 as a subalgebra such that 
$V=\g g/\g h$
 (Theorem \ref{thrhrqr}).  
\end{remark}


\subsection{Real and complex Lie algebras}
\label{secreacom}

We see from Lemma \ref{lemrhrvc} below
 that the second statement of Theorem \ref{thrhrq} follows from Theorem \ref{thrhrqc}. 
We recall that a real semisimple Lie algebra is split
if its real rank and complex rank coincide.

\begin{lemma}
\label{lemrhrvc}
Let $\g h$ be a real semisimple Lie algebra, 
$V$ a finite-dimensional representation of $\g h$, 
 and $V_{\m C}$ the complexification of $V$.  
\\
{\rm{(1)}}\enspace
Assume that $V$ has RGS in $\g h$
(Definition \ref{defrgs}).  
Then one has the equivalence: 
\[
\text{$V$ has $\rm{AGS}$ in $\g h$ $\Longleftrightarrow$
$V_\m C$ has $\rm{AGS}$ in $\g h_{\m C}$.}
\]
{\rm{(2)}}\enspace
 One has the implication:
\begin{eqnarray}
\rho_{\g h_\m C}\leq \rho_{V_\m C}
&\Longrightarrow &
\rho_{\g h}\leq \rho_V.
\end{eqnarray}
Moreover, the converse is true when $\g h$ is a split real semisimple
Lie algebra.
\end{lemma}

The proof of Lemma \ref{lemrhrvc} is straightforward 
 and is left to the reader. 

According to Proposition \ref{prorhrv}, and to Lemma \ref{lemrhrvc}, 
the following Conjec\-ture \ref{conrhrvc} 
is equivalent to Conjecture \ref{conrhrv}.

\begin{conjecture}
\label{conrhrvc}
Let  $\g h$ be a complex semisimple Lie algebra
and $V$ an orthogonal representation of $\g h$ over $\m C$. 
One has the  equivalence:
\begin{eqnarray}
\rho_{\g h}\leq \rho_V
&\Longleftrightarrow &
\mbox{$V$ has $\rm{AGS}$ in $\g h$.}
\end{eqnarray}
\end{conjecture}

Since the direct implication $\Longrightarrow$ follows from Proposition 
\ref{prorhrv}, one only has to understand the converse implication
$\Longleftarrow$.

\subsection{Representations of nonsimple Lie algebras}
\label{secmodnonsim}
The following Lemma \ref{lemv1v2} gives useful upper bounds 
for the invariant $p_V$ when the semisimple Lie algebra 
$\g h$ is not simple.
We collect some basic properties
 of the function $\rho_V$ \eqref{eqnrhoV}
 and the invariant $p_V$ \eqref{eqnpvinf}
 for representation $V$ of ${\mathfrak{h}}$.

\begin{lemma}
\label{lemv1v2}
Let $\g h=\g h_1\oplus \g h_2$ be a real semisimple Lie algebra, 
 which is the direct sum 
of two ideals $\g h_1$, $\g h_2$ and $V$ a finite-dimensional
representation of $\g h$.\\
{\rm{(1)}} For all $X_1$ in $\g h_1$ and $X_2$ in $\g h_2$, one has
\begin{eqnarray}
\label{eqnrvx1x2}
\rho_V(X_1)&\leq& \rho_V(X_1+X_2).
\end{eqnarray}
{\rm{(2)}} Assume that $V=V_1\otimes V_2$ where, for $i=1$, $2$, 
$V_i$ are representations of $\g h_i$ of dimension $d_i$. Then one has
\begin{eqnarray}
\label{eqnppdpd}
p_V\leq\tfrac{p_{V_1}}{d_2}+\tfrac{p_{V_2}}{d_1}.
\end{eqnarray} 
{\rm{(3)}} Assume now that $V=V_1\oplus V_2$ where, for $i=1$, $2$, 
$V_i$ are representations of $\g h$  such that 
$\rho_{\g h_i}\leq \rho_{V_i}$. Then one has
$\rho_{\g h}\leq  \rho_{V}.$
\end{lemma}

\begin{proof}
(1)
Let $a$ be an eigenvalue of $X_1$ in $V$, 
 and $b_1$, $\dots$, $b_r$ eigenvalues of $X_2$
 in $\operatorname{Ker}(X_1-a \operatorname{id}_V)$
 counted with multiplicities.  
Since $\g h_2$ is semisimple, 
 one has $\sum_{j=1}^r b_j=0$.  
In turn, 
 $r a = \sum_{j=1}^r (a+b_j)$, 
 yielding
\[
   \dim \operatorname{Ker} (X_1-a \operatorname{id}_V) |\operatorname{Re} a|
   \le 
   \sum_{j=1}^r |\operatorname{Re}(a+b_j)|.  
\]
Hence $\rho_V(X_1) \le \rho_V(X_1+ X_2)$.  
\newline
(2) Take any $X=X_1+X_2 \in \g h = \g h_1 \oplus \g h_2$.  
By the first statement, 
one has 
\begin{eqnarray*}
\rho_{\g h}(X)
& = &
\rho_{\g h_1}(X_1)+\rho_{\g h_2}(X_2)
\;\le\; 
p_{V_1}\rho_{V_1}(X_1)+p_{V_2}\rho_{V_2}(X_2)\\
&\leq&
\tfrac{p_{V_1}}{d_2}\rho_{V}(X_1)+\tfrac{p_{V_2}}{d_1}\rho_{V}(X_2)
\;\leq\;
\left( \tfrac{p_{V_1}}{d_2}+\tfrac{p_{V_2}}{d_1}\right)\rho_V(X).
\end{eqnarray*}
Hence the second statement follows.
\newline
(3)
 Take any $X=X_1+X_2 \in \g h = \g h_1 \oplus \g h_2$.  
By the first statement, 
 one has 
\[
   \rho_V(X_1+X_2)
  = 
  \rho_{V_1}(X_1+X_2)+\rho_{V_2}(X_1+X_2)
  \ge 
  \rho_{V_1}(X_1)+\rho_{V_2}(X_2), 
\]
 whereas $\rho_{\g h}(X_1+X_2) = \rho_{\g h_1}(X_1) + \rho_{\g h_2}(X_2)$.  
Hence the third statement follows.  
\end{proof}


\subsection{Reduction to simple Lie algebra}
\label{secredsim}

A real semisimple Lie algebra $\g h$ is said to be {\it{${\rm Ad}$-compact}} if 
the group of automorphisms ${\rm Aut}(\g h)$ is compact.
We denote by $\g h_{nc}$ the sum of the ideals of $\g h$
 which are not $\operatorname{Ad}$-compact.

The following Lemma \ref{lemrhrqreacom} allows us to assume the
reductive Lie sub\-algebra to be semisimple without {\rm Ad}-compact ideals.

\begin{lemma}
\label{lemrhrqreacom}
Let $\g g$ be a real semisimple Lie algebra,
$\g h$ a reductive Lie subalgebra of $\g g$, 
 and $\g s$
the semisimple Lie algebra $\g s:=[\g h,\g h]_{nc}$.
One has the equivalences~:
\begin{equation*}
\rho_{\g h}\leq \rho_{\g g/\g h} 
\;\Longleftrightarrow\;
\rho_{\g s}\leq\rho_{\g g/\g s}, 
\end{equation*}
\begin{equation*}
\mbox{$\g g/\g h$ has $\rm{AGS}$ in $\g h$}
\Longleftrightarrow
\mbox{$\g g/\g s$ has $\rm{AGS}$ in $\g s$}, 
\end{equation*}
\begin{equation*}
\mbox{$\g g/\g h$ has AmGS in $\g h$}
\Longleftrightarrow
\mbox{$\g g/\g s$ has AmGS in $\g s$}.  
\end{equation*}
\end{lemma}


The proof of Lemma \ref{lemrhrqreacom} is left to the reader.

The following Proposition \ref{prorhrqreacom} tells us that,
in order to prove Theorem \ref{thrhrqc}, we can assume 
$\g g$ to be simple.

\begin{proposition}
\label{prorhrqreacom}
Let $\g g$ be a real semisimple Lie algebra,
 $\g h$ a semisimple Lie subalgebra of $\g g$, $\g q:=\g g/\g h$, 
 $\g g=\g g_1\oplus\cdots\oplus \g g_r$
 a decomposition into ideals $\g g_j$,
and, for $1 \le i\leq r$, $\g h_i:=\g h\cap \g g_i$
and $\g q_i:=\g g_i/\g h_i$.
One has the equivalences~:
\begin{enumerate}
\item[{\rm{(1)}}]
\begin{equation*}
\rho_{\g h}\leq \rho_{\g q} \quad \text{on $\g h$}
\;\Longleftrightarrow\;
\rho_{\g h_i}\leq\rho_{\g q_i}\quad \text{on $\g h_i$}
\;\;\mbox{for all $1\leq i \leq r$};
\end{equation*}
\item[{\rm{(2)}}]
\begin{equation*}
\mbox{$\g q$ has $\rm{AGS}$ in $\g h$}
\Longleftrightarrow
\mbox{$\g q_i$ has $\rm{AGS}$ in $\g h_i$, for all $1 \leq i\leq r$};
\end{equation*}
\item[{\rm{(3)}}]
\begin{equation*}
\mbox{$\g q$ has {\rm{AmGS}} in $\g h$}
\Longleftrightarrow
\mbox{$\g q_i$ has {\rm{AmGS}} in $\g h_i$, for all $1 \leq i\leq r$}.  
\end{equation*}
\end{enumerate}
\end{proposition}


Before giving a proof of Proposition \ref{prorhrqreacom},
we set up some notation.  
We write $\pi_i \colon \g g \to \g g_i$
 for the $i$-th projection ($1 \le i \le r$).  
Given a subspace $V$ in $\g g_1 \oplus \cdots \oplus \g g_r$, 
 we define the \lq\lq{hull of $V$}\rq\rq\ by $\widetilde V := \pi_1(V) \oplus \cdots \oplus \pi_r(V)$.

For each $\sigma \in \operatorname{Map}(\{1,2,\dots,r\}, \{+,-\})$, 
 we define a vector space $V^{\sigma}$ by
\[
   V^{\sigma}:=\{(\sigma(1)v_1, \dots, \sigma(r)v_r)\,|\, (v_1, \dots, v_r) \in V\}.  
\]
Then $\widetilde V=\sum_{\sigma} V^{\sigma}$
 where the sum is taken over all $\sigma$.  
We note that $V \subsetneq \widetilde V$
 if and only if $V \cap \g g_i \subsetneq \pi_i(V)$
 for some $1 \le i \le r$, 
 or equivalently,
 $V \ne V^{\sigma}$ for some $\sigma$.

If $V$ is a semisimple Lie algebra, 
 then so is $\widetilde V$
 because $\pi_i(V)$s are semisimple ideals.  
If $[V',V'']=\{0\}$, 
 then $[\widetilde{V'},\widetilde{V''}]=\{0\}$.  
In particular, 
 if the Lie algebra $V$ is a direct sum
 of two semisimple ideals $V'$ and $V''$, 
 then its hull $\widetilde V$ is also a direct sum of semisimple ideals $\widetilde{V'}$
 and $\widetilde{V''}$,  
\begin{equation}
\label{eqn:vvpvpp}
\widetilde{V}=\widetilde{V'}\oplus \widetilde{V''}.
\end{equation}

\begin{proof}
[Proof of Proposition \ref{prorhrqreacom}]
For a nonempty set $I \subset \{1,\dots,r\}$, 
 we define an ideal $\g h_I$ of $\g h$ inductively
 on the cardinality $\# I$ of $I$ 
 by
\[
\text{$\g h_I:=\g h_i= \g h \cap \g g_i$
\quad when $I=\{i\}$\, $(1 \le i \le r)$ }
\]
and by the following characterization:
\[
\text
{
$\g h \cap (\bigoplus_{i \in I} \g g_i)= \g h_I \oplus(\bigoplus_{J \subsetneq I} \g h_J)$ when $\#I \ge 2$.  }
\]
Then one sees readily from the definition of $\g h_I$:
\begin{align}
\label{eqn:hhIdeco}
&\g h =\bigoplus_I \g h_I
\qquad
\text{(direct sum of semisimple ideals),}
\\
\label{eqn:2017102}
&\g h_I \cap (\bigoplus_{j \in J} \g g_j)=\{0\}
\quad
\text{if $I \not \subset J$.}
\end{align}
In particular,
\begin{equation}
\label{eqn:2017103}
\g h_I \cap (\g h_I)^{\sigma}=\{0\}
\quad
\text{for any $\sigma$ with $\sigma|_I \ne \pm {\bf 1}_I$.}
\end{equation}
Here $\sigma|_I \ne\pm {\bf 1}_I$ means
 that $\sigma(i)\ne\sigma(j)$ for some $i,j \in I$.

We choose an $\g h_I$-submodule $\g q_I$ in $\widetilde{\g h_I}$
with a direct sum decomposition
\[
  \widetilde{\g h_I}=\g h_I\oplus\g q_I.  
\]
We note that $\g q_I=\{0\}$ when $\# I=1$.  
By \eqref{eqn:2017103}, 
 if $\# I \ge 2$, 
 then we may and do take $\g q_I$
 to contain the $\g h_I$-submodule $(\g h_I)^{\sigma}$
 for some $\sigma$.  
Since $(\g h_I)^{\sigma} \simeq \g h_I$ as $\g h_I$-modules, 
 this implies that if $\# I \ge 2$, 
\begin{align}
\label{eqn:hIqI}
&\bullet\enspace \rho_{\g h_I} \le \rho_{\g q_I}\quad \text{on $\g h_I$}, 
\\
\label{eqn:qIAGS}
&\bullet\enspace \text{$\g q_I$ has $\rm{AGS}$ in $\g h_I$.  }
\end{align}
Moreover, 
 if $i \in I$, 
 one has
\[
  \pi_i(\g h_I)=
\begin{cases}
\g h_i
\quad
&\text{when $\# I =1$,}
\\
\pi_i(\g q_I)
\quad
&\text{when $\# I \ge 2$.  }
\end{cases}
\]

By \eqref{eqn:vvpvpp} and \eqref{eqn:hhIdeco}, 
 one has
\[
  \widetilde {\g h}= \bigoplus_I \widetilde{\g h_I}
\quad
\text{(direct sum of semisimple ideals).}
\]
Taking the projection to the $i$-th component,
 one obtains $\pi_i(\g h)=\oplus_I \pi_i(\g h_I)$, 
 hence
\begin{equation}
\label{eqn:201794}
  \pi_i(\g h)=\g h_i \oplus \pi_i(\bigoplus_{\# I \ge 2} \g q_I).  
\end{equation}

For each $i$ ($1 \le i \le r$), 
 we write 
\[
   \g g_i= \pi_i (\g h) \oplus \g s_i
\]
by taking a $\pi_i(\g h)$-invariant subspace $\g s_i$ in $\g g_i$, 
 and set 
\begin{align}
\notag
\g s :=& \g s_1 \oplus \cdots \oplus \g s_r,
\\
\label{eqn:qSqI}
\g q :=& \g s \oplus (\bigoplus_{\# I \ge 2} \g q_I).  
\end{align}
Then $\g q \simeq \g g/\g h$ as an $\g h$-module
 because one has the following direct sum decompositions:
\begin{equation*}
\g g= \widetilde{\g h} \oplus \g s
=(\bigoplus_I \widetilde {\g h_I}) \oplus \g s
=\bigoplus_I \g h_I \oplus \bigoplus_I \g q_I \oplus \g s
=\g h \oplus \g q.  
\end{equation*}
Moreover, 
 \eqref{eqn:201794} tells
\begin{equation}
\label{eqn:gihipiq}
  \g g_i = \pi_i(\g h) \oplus \g s_i
         = \g h_i \oplus \g s_i \oplus \pi_i(\bigoplus_{\# I \ge 2} \g q_I)
         = \g h_i \oplus \pi_i(\g q).  
\end{equation}
In particular, 
 $\g q_i =\g g_i /\g h_i$ is expressed
 as an $\g h_i$-module:
\begin{equation}
\label{eqn:qisihi}
  \g q_i \simeq \g s_i \oplus \text{(trivial $\g h_i$-module).  }
\end{equation}
\newline
(1)
 Suppose $\rho_{\g h} \le \rho_{\g q}$. 
Then for any $H \in \g h_i$, 
\[
   \rho_{\g h_i}(H) \le  \rho_{\g h}(H)
  \le \rho_{\g q} (H)
  = \rho_{\g s}(H) + \sum_{\# I \ge 2}\rho_{\g q_I}(H)
  =  \rho_{\g s_i}(H)
\]
because $\g h_i$ acts trivially on all $\g s_j$
 with $j \ne i$
 and $\g q_I$ with $\#I \ge 2$.

Conversely, 
 suppose $\rho_{\g h_i} \le  \rho_{\g q_i}$ holds
 for all $1 \le i \le r$.  
Take any $H \in \g h$, 
 and write 
\[
  H= \sum_I H_I = \sum_{i=1}^r H_i + \sum_{\#I\ge 2}H_I 
  \in \g h =\bigoplus_{i=1}^r \g h_i \oplus \bigoplus_{\#I \ge 2} \g h_I.  
\]
Then 
\[
   \rho_{\g h} (H) 
  =\sum_I \rho_{\g h_I}(H)
  =\sum_I \rho_{\g h_I}(H_I)
  =\sum_{i=1}^r \rho_{\g h_i}(H_i) + \sum_{\# I \ge 2} \rho_{\g h_I}(H_I).
\]
By the assumption $\rho_{\g h_i}(H_i) \le \rho_{\g q_i}(H_i)$
 and by \eqref{eqn:qisihi}
 and \eqref{eqn:hIqI}, 
 one obtains
\[
   \rho_{\g h}(H) \le \sum_{i=1}^r \rho_{\g s_i}(H_i) + \sum_{\# I \ge 2} \rho_{\g q_I}(H_I)
    =
    \rho_{\g s}(\sum_{i=1}^r H_i) + \sum_{\# I \ge 2} \rho_{\g q_I}(H).
\]

By Lemma \ref{lemv1v2} (1), 
 one has $\rho_{\g s}(\sum_{i=1}^r H_i) \le  \rho_{\g s}(H)$, 
 hence $\rho_{\g h}(H) \le  \rho_{\g q}(H)$.  
\par\noindent
(2) 
Suppose $\g q$ has $\rm{AGS}$ in $\g h$.  
Let $U$ be a dense subset of $\g q$ 
such that $\operatorname{Stab}_{\g h} (x) \equiv \g h_x$
 is abelian and reductive for all $x \in U$.  
Then for all $1 \le i \le r$, 
 $\operatorname{Stab}_{\g h_i} (\pi_i(x)) = \operatorname{Stab}_{\g h_i} (x)$
 is abelian 
 and reductive.  
Since $\pi_i(U)$ is dense in $\g q_i = \g g_i/\g h_i \simeq \pi_i(\g q)$
 by \eqref{eqn:gihipiq}, 
 $\g q_i$ has $\rm{AGS}$ in $\g h_i$.

Conversely, 
 suppose $\g q_i$ has $\rm{AGS}$ in $\g h_i$
 for all $1 \le i \le r$.  
By \eqref{eqn:qisihi}, 
 $\g s_i$ has also $\rm{AGS}$ in $\g h_i$.  
By \eqref{eqn:qIAGS}, 
 one can find a dense subset $W$
 of $\g q= \underset{i=1}{\overset r\bigoplus} \g s_i \oplus \underset {\# I \ge 2}\bigoplus \g q_I$, 
 see \eqref{eqn:qSqI}, 
 such that if $x=\sum_{i=1}^r y_i + \sum_{\# I \ge 2} z_I \in W$
 then $\operatorname{Stab}_{\g h_i} (y_i)$ ($1 \le i \le r$)
 and $\operatorname{Stab}_{\g h_I} (z_I)$ ($\# I \ge 2$)
 are all abelian and reductive.  
We now observe
\begin{align*}
   \operatorname{Stab}_{\g h} (x)
 =&  \bigcap_{i=1}^r \operatorname{Stab}_{\g h} (y_i)
     \cap 
     \bigcap_{\#I \ge 2} \operatorname{Stab}_{\g h} (z_I)
\\
 =& \bigcap_{i=1}^r \operatorname{Stab}_{\g h} (y_i)
    \cap 
    (\bigoplus_{i=1}^r \g h_i \oplus \bigoplus_{\# I \ge 2} 
    \operatorname{Stab}_{\g h_I} (z_I)).  
\end{align*}
Therefore the (splitting) exact sequence 
 $0 \to \oplus_{i=1}^r \g h_i \to \g h \to \oplus_{\#I \ge 2} \g h_I \to 0$
 induces an exact sequence of Lie algebras:
\[
  0 \to \bigoplus_{i=1}^r \operatorname{Stab}_{\g h_i} (y_i)
  \to \operatorname{Stab}_{\g h} (x) 
  \to \bigoplus_{\# I \ge 2} \operatorname{Stab}_{\g h_I} (z_I).  
\]
By Proposition \ref{proortredsta}, the Lie algebra $\operatorname{Stab}_{\g h} (x)$ is reductive for $x$ in an open dense subset of $\g q$. The above exact sequence tells us that it is also abelian.
\par\noindent
(3) The proof parallels to that of (2).  
\end{proof}

\section{Classical simple Lie algebras}
\label{seccla}
In this chapter  we give 
a classification of the pairs $(\g g, \g h)$
of complex semisimple Lie algebras 
satisfying $\rho_{\g h} \not \le \rho_{\g g/\g h}$
in the case 
where $\g g$ is classical simple, 
 and in particular,  
 prove Theorems \ref{thrhrqc} and \ref{thghtempc}
for $\g g$ classical simple.

Throughout this chapter, 
$\g g$ is a complex classical simple Lie algebra,
$\g h$ is a complex semisimple Lie subalgebra
$\{0\} \neq \g h \subsetneq \g g$, $\g q:=\g g/\g h$,
$\wt{\g h}$ is the normalizer of $\g h$ in $\g g$,
$\g m$ is the generic stabilizer of $\g q$ in $\g h$,
and $\g m_s:=[\g m,\g m]$.
``Classical'' means that 
$\g g=\g s\g l(\m C^n)$, $\g s\g o(\m C^n)$ or $\g s\g p(\m C^{2n})$.
We will denote by $V$ the standard representation of $\g g$ in $\m C^n$,
$\m C^n$ or $\m C^{2n}$ respectively.

\subsection{Main list for classical Lie algebras}
\label{secmaicla}

We will use the notations $\g s\g l_n$, $\g s\g o_n$ and $\g s\g p_n$ for
$\g s\g l(\m C^n)$, $\g s\g o(\m C^n)$ and $\g s\g p(\m C^{2n})$
and also $\g a_\ell$, $\g b_\ell$, $\g c_\ell$, $\g d_\ell$
for $\g s\g l_{\ell+1}$, $\g s\g o_{2\ell +1}$, $\g s\g p_{\ell}$, 
$\g s\g o_{2\ell}$, and $\g g_2$, $\g f_4$, $\g e_6$, $\g e_7$, $\g e_8$
for the five exceptional simple Lie algebras.

\begin{theorem}
\label{thabcd}
Let $\g g=\g s\g l_n$, $\g s\g o_n$ or $\g s\g p_n$
be a complex classical simple Lie algebra. 
The complex semisimple Lie subalgebras 
$\g h\subsetneq \g g$ satisfying 
$\rho_\g h\not\leq\rho_\g q$ form the list in Table \ref{figabcd}. 
In this list, $\g q$ does not have $\rm{AGS}$ in $\g h$.
\end{theorem}

\begin{figure}[htb]
\begin{center}
\begin{tabular}{|c|c|c|c|c|c|c|}
\hline
\small\!\!Case\!\!&$\g g$&\small\!max. \!$\g h$\!&$\g m$&\small\!\!non\! max.$\g h$\!\!&$\g m$&\small parameters\\
&&&&&&\small  $p \ge q \ge 1$\\
\hline
$A1$&$\g s\g l_{p+q}$&$\g s\g l_p\!\oplus\!\g s\g l_q$&
$\g s\g l_{p-q}\!\oplus\!\m C^{q-1}$&
$\g s\g l_{p}\!\oplus\! \g h_2$
&\!$\supset \g s\g l_{p-q}$\!&\!$p\!\geq\! q\!+\!2$ $\g h_2\!\subset\!\g s\g l_q$\!\\
\hline
$A2$&$\g s\g l_{2p}$&$\g s\g p_p$&
$(\g s\g l_2)^p$&&&
$p\geq 2$\\
\hline
$\!\!B\!D1\!\!$&$\g s\g o_{p+q}$&\!$\g s\g o_p\!\oplus\!\g s\g o_q$\!&
$\g s\g o_{p-q}$&
$\g s\g o_{p}\!\oplus\! \g h_2$&
$\g s\g o_{p-q}$&\!\!\begin{tabular}{c} 
$p\!\geq\! q\!+\!3$\\
$q\neq 2$
\end{tabular}\!$\g h_2\!\subset\!\g s\g o_q$\!\!\\
&$\g s\g o_{p+2}$&&&
$\g s\g o_{p}$&
$\g s\g o_{p-2}$&$p\!\geq\! 5$\\
$D4$&$\g s\g o_{7+1}$&&&
$\g g_2$& 
$\g s\g l_2$&
$\g g_2\stackrel{irr}{\hookrightarrow} \g s\g o_7$\\
$B4$&$\g s\g o_{8+1}$&&&
$\g s\g o_7$&
$\g s\g l_3$& 
$\g s\g o_7\stackrel{irr}{\hookrightarrow} \g s\g o_8$\\
$D5$&$\g s\g o_{8+2}$&&&
$\g s\g o_7$&
$\g s\g l_2$& 
$\g s\g o_7\stackrel{irr}{\hookrightarrow} \g s\g o_8$\\
\hline
$D2$&$\g s\g o_{2p}$&$\g s\g l_p$&
$\supset (\g s\g l_2)^{[p/2]}$&&&
$p\geq 3$\\
\hline
$B3$&$\g s\g o_{7}$&$\g g_2$&
$\g s\g l_3$&&&
$\g g_2\stackrel{irr}{\hookrightarrow} \g s\g o_7$\\
\hline
$C1$&$\g s\g p_{p+q}
$\!&\!$
\g s\g p_p\!\oplus\!\g s\g p_q
$\!&\!\!$
\g s\g p_{p-q}\!\oplus\!(\g s\g p_1)^p$\!\!&\!$\g s\g p_{p}\!\oplus\! \g h_2$
&\!$\supset \g s\g p_{p-q}$\!&\!$p\!\geq\! q\!+\!1$ $\g h_2\!\subset\!\g s\g p_q$\!\\
\hline
$C2$&$\g s\g p_{2p}
$&\!$\g s\g p_p\!\oplus\!\g s\g p_p$\!&
$(\g s\g p_1)^p$&&&
$p\geq 1$\\
\hline
\end{tabular}
\end{center}
\caption{
Pairs $(\g g,\g h)$ with $\rho_{\g h}\not\leq\rho_{\g q}$
for $\g g$ classical simple}
\label{figabcd}
\end{figure}

The left-hand side of Table \ref{figabcd} lists the semisimple Lie subalgebras
$\g h\subsetneq \g g$  which are maximal (among the semisimple 
Lie subalgebras of $\g g$), 
while the right-hand side lists non-maximal ones.  
Note that when a maximal ${\mathfrak{h}}$ does not contain 
a proper semisimple subalgebra ${\mathfrak{h'}}$
with $\rho_{\mathfrak{h'}} \not \le \rho_{\mathfrak{q'}}$,
one has a blank in the right-hand side (A2, D2, B3, C2).
The blanks on the left-hand side 
(the second case of BD1, D4, B4, D5)
means that the non-maximal ${\mathfrak{h}}$ is 
a subalgebra of a maximal semisimple subalgebra ${\mathfrak{h'}}$
which already occurred in another row (BD1 with $q=1$).


Note that in Table \ref{figabcd},
the pair $(\g s\g o_7,\g g_2)$ 
is the only one for which $\g h$ is maximal and 
$(\g g,\wt{\g h})$ is not a symmetric pair.

In case $D2$, the morphisms $\g s\g l_p\hookrightarrow\g s\g o_{2p}$ are 
those for which $\wt{\g h}$ are the stabilizers
of a pair of transversal isotropic $p$-planes 
in $\m C^{2p}$.

In Cases $B3$ and $D4$, 
the morphisms $\g g_2\hookrightarrow\g s\g o_n$ $(n=7,8)$ are given by the 
$7$-dimensional irreducible representation 
$\g g_2\stackrel{irr}{\hookrightarrow}\g s\g o_7$, 
plus $n\!-\!7$ copies of the trivial one-dimensional representation.

In Cases $B4$ and $D5$, 
the morphisms $\g s\g o_7\hookrightarrow\g s\g o_n$ $(n=9,10)$ are given by the 
$8$-dimensional irreducible representation
$\g s\g o_7\stackrel{irr}{\hookrightarrow}\g s\g o_8$, 
called the spin repre\-sentation, 
plus $n\!-\!8$ copies of the trivial one-dimensional representation.
Note that the pair $\g s\g o_7\stackrel{irr}{\hookrightarrow}\g s\g o_8$ itself 
is not included in the left-hand side of Table \ref{figabcd}, because 
it is isomorphic to the standard pair $\g s\g o_7\subset\g s\g o_8$ by an outer 
automorphism of $\g s\g o_8$.
\vspace{1em}

The strategy of the proof of Theorem \ref{thabcd} is to deal first with natural examples of pairs 
$(\g g,\g h)$ where $\g h$ is  maximal in $\g g$~:
for symmetric pairs in Section \ref{secclasym},
for irreducible representations in Section \ref{secclairr},
and for reducible representations in Section \ref{secclared1}. 
The proof of Theorem \ref{thabcd} is given in Section \ref{secclapro}, 
except that most of the technical estimates  
are postponed to Chapter \ref{secbpvnon}.  

\subsection{Classical symmetric pairs}
\label{secclasym}

We first deal with the seven families
of pairs $(\g g, \g h)$
such that $(\g g,\wt{\g h})$ is a classical symmetric pair, 
 where $\wt{\g h}$ is the normalizer of $\g h$ in $\g g$.
We give a necessary and sufficient condition 
 for $\rho_{\g h} \not \le \rho_{\g q}$.  
We also list the generic stabilizer $\g m$, 
which is readily computed by using the Satake diagram of 
the structure theory of symmetric pairs $(\g g, \widetilde {\g h})$, 
see \cite[Chap.~10]{He78} for instance.

\begin{proposition}
\label{proclasym1}
Let $p\geq q\geq 1$. \\
$\bullet$ If $\g g=\g s\g l_{p+q}\supset \g h=\g s\g l_p\oplus\g s\g l_q$,
then $\g m\simeq \g s\g l_{p-q}\oplus \m C^q$ and
$\rho_\g h\not\leq \rho_\g q\Leftrightarrow |p-q|\geq 2$.\\
$\bullet$ If $\g g=\g s\g o_{p+q}\supset \g h=\g s\g o_p\oplus\g s\g o_q$,
then $\g m\simeq \g s\g o_{p-q}$ and
$\rho_\g h\not\leq \rho_\g q\Leftrightarrow |p-q|\geq 3$.\\
$\bullet$ If $\g g=\g s\g p_{p+q}\supset \g h=\g s\g p_p\oplus\g s\g p_q$,
then $\g m\simeq \g s\g p_{p-q}\oplus (\g s\g p_1)^{q}$
 and
$\rho_\g h\not\leq \rho_\g q$.
\end{proposition}

\begin{proof}
This follows from Proposition \ref{proclasym1v} because in these examples, one has respectively
$\g q={\mathbb{C}} \oplus (V \oplus V^{\ast})$, 
$\g q=V$ and $\g q=V$ where $V=\m C^p\otimes\m C^q$
 or $\m C^{2p} \otimes \m C^{2q}$.
\end{proof}

\begin{proposition}
\label{proclasym2}
Let $p\geq 1$  and set $\ell:=[\tfrac{p}{2}]$, $\eps:=p-2\ell \in \{0,1\}$.
\\
$\bullet$ If $\g g=\g s\g l_{p}\supset \g h=\g s\g o_p$,
then $\g m= \{0\}$ and
$\rho_\g h\leq \rho_\g q$.\\
$\bullet$ If $\g g=\g s\g l_{2p}\supset \g h=\g s\g p_p$,
then $\g m\simeq (\g s\g l_2)^p$ and
$\rho_\g h\not\leq \rho_\g q$.\\
$\bullet$ If $\g g=\g s\g o_{2p}\supset \g h=\g s\g l_p$,
then $\g m\simeq (\g s\g l_2)^{\ell}\oplus \m C^{\eps}$ and
$\rho_\g h\not\leq \rho_\g q$.\\
$\bullet$ If $\g g=\g s\g p_{p}\supset \g h=\g s\g l_p$,
then $\g m =\{0\}$ and
$\rho_\g h\leq \rho_\g q$.
\end{proposition}

\begin{proof}
This follows from Propositions \ref{probpval}, \ref{probpvbl}, \ref{probpvcl} and \ref{probpvdl} because in these examples,
 one has respectively
$\g q=S^2_0\m C^p \simeq S^2 {\mathbb{C}}^p / {\mathbb{C}}$, 
$\g q=\La^2_0\m C^{2p}\simeq \Lambda^2 {\mathbb{C}}^{2p} / {\mathbb{C}}$, 
$\g q=\m C\oplus(\La^2\m C^p\oplus \text{dual})$, and 
$\g q=\m C\oplus(S^2\m C^p\oplus \text{dual})$.
\end{proof}

\subsection{Irreducible representations}
\label{secclairr}

In this section we deal with semisimple Lie subalgebras $\g h$
of $\g g=\g s\g l_n$, $\g s\g o_n$ or $\g s\g p_n$ whose action 
on $V=\m C^n$, $\m C^n$ or $\m C^{2n}$ is irreducible.

The first proposition deals with the case when $\g h$ is not simple,
{\it{i.e.}} $\g h$ is the sum of two non-zero ideals ${\mathfrak{h}}_1$ and ${\mathfrak{h}}_2$, 
{\it{i.e., }} ${\mathfrak{h}}={\mathfrak{h}}_1 \oplus {\mathfrak{h}}_2$.  
Then $\g h$ is realized as a subalgebra of $\g g$
 via the outer tensor product 
of the natural representations of $\g h_1$ and $\g h_2$.

\begin{proposition} 
\label{proclaten}
Suppose $p> 1$ and $q>1$.  
\mbox{ }\\
$\bullet$ If $\g g=\g s\g l_{pq}\supset\g h=\g s\g l_p\oplus\g s\g l_q$,
then $\g m=\{0\}$ and $\rho_\g h\leq\rho_\g q$.\\
$\bullet$ If $\g g=\g s\g o_{pq}\supset\g h=\g s\g o_p\oplus\g s\g o_q$,
then $\g m=\{0\}$ and $\rho_\g h\leq\rho_\g q$.\\
Suppose $p\geq 1$ and $q>1$.\\
$\bullet$ If $\g g=\g s\g o_{4pq}\supset\g h=\g s\g p_p\oplus\g s\g p_q$,
then $\g m=\{0\}$ and $\rho_\g h\leq\rho_\g q$.\\
$\bullet$ If $\g g=\g s\g p_{pq}\supset\g h=\g s\g p_p\oplus\g s\g o_q$,
then $\g m=\{0\}$ and $\rho_\g h\leq\rho_\g q$.
\end{proposition}

The computation of the generic stabilizers $\g m$ is straightforward, 
and the inequality $\rho_{\g h} \le \rho_{\g q}$ in Proposition \ref{proclaten} follows from Proposition \ref{proclatenv}.

\begin{proposition}
\label{proclairr}
Let $\g g=\g s\g l_n$, $\g s\g o_n$ or $\g s\g p_n$ and $\g h\subsetneq \g g$ a simple Lie subalgebra which is irreducible on $V$ and 
satisfies $\rho_\g h\not\leq\rho_\g q$.\\
$\bullet$ If $\g g=\g s\g l_n$, then $n=2p$, $\g h=\g s\g p_{p}$
and 
$\g m\simeq(\g s\g l_2)^p$.\\
$\bullet$ If $\g g=\g s\g o_n$, then $n=7$, $\g h=\g g_2$, $\g m\simeq \g s\g l_3$
or $n=8$, $\g h=\g s\g o_7$, $\g m\simeq \g s\g o_6$.\\
$\bullet$ If  $\g g=\g s\g p_n$ then such an $\g h$ does not exist. 
\end{proposition}

The proof of Proposition \ref{proclairr} relies on explicit 
computations of $\rho_\g h$ and $\rho_\g q$.

\subsection{Example of reducible representations}
\label{secclared1}

In this section, we deal with semisimple Lie subalgebras $\g h$
of $\g g=\g s\g l_n$, $\g s\g o_n$ or $\g s\g p_n$ whose action 
on $V=\m C^n$, $\m C^n$ or $\m C^{2n}$ is reducible.
We have already discussed those subalgebras $\g h$ which are maximal in $\g g$
in Propositions \ref{proclasym1} and \ref{proclasym2}.
We focus on the most important examples for which $\g h$ is not maximal.

The first proposition deals mainly with the case 
where the vector space $V$ has more than two irreducible components.

\begin{proposition}
\label{proslslsl}
Let $r\geq 1$, $n\geq n_1+\cdots +n_r$ with $n_1\geq\cdots\geq n_r\geq 1$.
\\
$\bullet$ If $\g g=\g s \g l_n\supset 
\g h=\g s\g l_{n_1}\oplus\cdots\oplus\g s\g l_{n_r}$, 
then\\ 
$\rho_\g h\not\leq \rho_\g q\Leftrightarrow 2n_1\geq n+2$.
In this case, one has $\g m_s\simeq \g s\g l_{2n_1-n}$.\\
$\bullet$ If $\g g=\g s \g o_n\supset 
\g h=\g s\g o_{n_1}\oplus\cdots\oplus\g s\g o_{n_r}$, 
then\\ 
$\rho_\g h\not\leq \rho_\g q\Leftrightarrow 2n_1\geq n+3$.
In this case, one has $\g m\simeq \g s\g o_{2n_1-n}$.\\
$\bullet$ If $\g g=\g s \g p_n\supset 
\g h=\g s\g p_{n_1}\oplus\cdots\oplus\g s\g p_{n_r}$, 
then\\ 
$\rho_\g h\not\leq \rho_\g q\Leftrightarrow 2n_1\geq n+1$
or $n=2n_1=2n_2$.\\
In this case, one has $\g m\supset \g s\g p_{2n_1-n}$ or $\g m\simeq (\g s\g p_1)^{n_1}$
respectively.
\end{proposition}

When $r=2$, Proposition \ref{proslslsl} is Proposition \ref{proclasym1}. 

When $r=3$, Proposition \ref{proslslsl} follows from Proposition  \ref{proslslslv}.

When $r\geq 4$, the proof 
 for the implication $\Rightarrow$ 
is  by induction on $r$ replacing the last two integers
$n_{r-1}$ and $n_r$ 
by their sum $n_{r-1}\!+\!n_r$ and reordering.

The opposite implication $\Leftarrow$ is easier.  
To see this, 
 let $\g h_1$ be the first factor of $\g h$, 
 and we set 
 $c=2,3$, 
 and $1$ for $\g g=\g s \g l_n$, $\g s \g o_n$, 
 and $\g s \g p_n$, 
respectively.  
Then  one computes 
\[
   p_{\g g/\g h_1}=\frac{n_1+1-c}{n-n_1}, 
\]
 by using \eqref{eqn:pvmult} and 
 Propositions \ref{probpval}, \ref{probpvbl}, 
 \ref{probpvdl}, and \ref{probpvcl}, 
 respectively.  
Hence
\begin{equation}
\label{eqn:simpleblock}
  \rho_{\g h_1} \not\le \rho_{\g g/\g h_1}
\quad
\text{if $2 n_1 \ge n+c$,}
\end{equation}
and thus the sufficiency of the inequality 
 in Proposition \ref{proslslsl} is shown.

\vspace{1em}

The second proposition deals mainly with the case 
where the vector space $V$ has two irreducible components.

\begin{proposition}
\label{proslspsl}
Let $p\geq 1$, $q\geq 1$.\\
$\bullet$ If $\g g=\g s \g l_{2p+q}\supset 
\g h=\g s\g p_{p}\oplus\g s\g l_{q}$, 
then 
$\rho_\g h\not\leq \rho_\g q\Leftrightarrow q\geq 2p+2$.\\
In this case, one has $\g m_s\simeq \g s\g l_{q-2p}$.\\
$\bullet$ If $\g g=\g s \g o_{2p+q}\supset 
\g h=\g s\g l_{p}\oplus\g s\g o_{q}$, 
then 
$\rho_\g h\not\leq \rho_\g q\Leftrightarrow q\geq 2p+3$.\\
In this case, one has $\g m\simeq \g s\g o_{q-2p}$.\\
$\bullet$ If $\g g=\g s \g p_{p+q}\supset 
\g h=\g s\g l_{p}\oplus\g s\g p_{q}$, 
then 
$\rho_\g h\not\leq \rho_\g q\Leftrightarrow q\geq p+1$.\\
In this case, one has $\g m_s\simeq \g s\g p_{q-p}$.\\
$\bullet$ If $\g g=\g s \g o_{4p}\supset 
\g h'=\g s\g l_{2p}\supset\g h=\g s\g p_{p}$ and $p\geq 2$, 
then one  has
$\rho_\g h\leq \rho_\g q$.
\end{proposition}

Proposition \ref{proslspsl} follows from Proposition \ref{proslspslv}.
Alternatively,
 the implication $\Leftarrow$ in Proposition \ref{proslspsl}
 follows readily from \eqref{eqn:simpleblock}.  
\vspace{1em}

The second proposition deals mainly with the case 
where the vector space $V$ has two irreducible components.

\begin{proposition}
\label{prososo}
Let  $q\geq 1$.\\
$\bullet$ If $\g g=\g s \g o_{7+q}\supset 
\g h=\g g_2\oplus\g s\g o_{q}$, 
then 
$\rho_\g h\not\leq \rho_\g q\Leftrightarrow q=1$ or $q\geq 10$.\\
In this case, one has $\g m\simeq \g s\g l_2$ or $\g m\simeq \g s\g o_{q-7}$.\\
$\bullet$ If $\g g=\g s \g o_{8+q}\supset 
\g h=\g s\g o_{7}\oplus\g s\g o_{q}$, 
then 
$\rho_\g h\not\leq \rho_\g q\Leftrightarrow q=1$, $q=2$, or $q\geq 11$.\\
In this case, one has 
$\g m\simeq \g s\g l_3$, $\g m\simeq\g s\g l_2$ or $\g m\simeq \g s\g o_{q-8}$.
\end{proposition}

Proposition \ref{prososo} follows from Proposition \ref{prososov}.

\subsection{Checking Theorem \ref{thabcd}}
\label{secclapro}

In this section, we check
Theorem \ref{thabcd}.
Let $V$ be $\m C^n$ or $\m C^{2n}$
 when $\g g= \mathfrak{s l}_n$
 and $\mathfrak{s o}_n$
 or $\g g= \mathfrak{s p}_n$, 
respectively.  
We have to deal now with pairs $(\g g,\g h)$
for which the action of $\g h$ on $V$ is reducible.

\begin{proposition}
\label{proslopn}
Let $p\geq q\geq 1$\\ 
$\bullet$ Let $\g g=\g s\g l_{p+q}$, and $\g h\subset \g g$ 
a semisimple Lie subalgebra 
included in $\g s\g l_p\oplus\g s\g l_q$ irreducible on $\m C^p$. 
One has the equivalence~:
\begin{equation}
\label{eqnslnslpslq}
\mbox{$\rho_\g h\not\leq\rho_\g q$}
\Longleftrightarrow
\mbox{$p\geq q+2$ and 
$\g h=\g s\g l_p\oplus \g h'$ with 
$\g h'\subset\g s\g l_{q}$.}
\end{equation}
$\bullet$ Let $\g g=\g s\g o_{p+q}$, and $\g h\subset \g g$ 
a semisimple Lie subalgebra 
included in $\g s\g o_p\oplus\g s\g o_q$, irreducible on $\m C^p$. 
One has the equivalence~:
\begin{equation}
\label{eqnsonsopsoq}
\mbox{$\rho_\g h\not\leq\rho_\g q$}
\Longleftrightarrow
\left\{\mbox{
\begin{minipage}{20em}
either $p\geq q+3$ and $\g h=\g s\g o_p\oplus \g h'$ with 
$\g h'\subset\g s\g o_{q}$;
\\
or $p=7$, $q=1$ and $\g h=\g g_2\stackrel{irr}{\hookrightarrow}\g s\g o_7$;
\\
or $p=8$, $q\leq 2$ and $\g h=\g s\g o_7
\stackrel{irr}{\hookrightarrow}\g s\g o_8$.
\end{minipage}}\right.
\end{equation}
$\bullet$ Let $\g g=\g s\g p_{p+q}$, and $\g h\subset \g g$ 
a semisimple Lie subalgebra 
included in $\g s\g p_p\oplus\g s\g p_q$, irreducible on $\m C^{2p}$. 
One has the equivalence~:
\begin{equation}
\label{eqnspnsppspq}
\mbox{$\rho_\g h\not\leq\rho_\g q$}
\Longleftrightarrow
\left\{\mbox{
\begin{minipage}{22em}
either $p\geq q+1$ and $\g h=\g s\g p_p\oplus \g h'$ with 
$\g h'\subset\g s\g p_{q}$;
\\
or $p=q$ and  $\g h=\g s\g p_p\oplus\g s\g p_p$.
\end{minipage}}\right.
\end{equation}
\end{proposition}

The implication $\Leftarrow$ is straightforward.  
To see the nontrivial implication $\Rightarrow$, 
 we observe 
 that $\rho_{\g k_2} \le \rho_{\g g/\g k}$
 on $\g k_2$
 as in \eqref{eqn:simpleblock}, 
 where $\g k= \g k_1 \oplus \g k_2$
 and $(\g g, \g k)
=({\mathfrak {s l}}_{p+q}, {\mathfrak{s l}}_p \oplus {\mathfrak {s l}}_q)$, 
 $({\mathfrak {s o}}_{p+q}, {\mathfrak{s o}}_p \oplus {\mathfrak {s o}}_q)$, 
 or 
$({\mathfrak {s p}}_{p+q}, {\mathfrak{s p}}_p \oplus {\mathfrak {s p}}_q)$
 with $p \ge q$.  
Then the implication $\Rightarrow$
 in Proposition \ref{proslopn} follows from  Lemma \ref{lemgh1h2} below
 and from the three previous Propositions 
 \ref{proclairr}, 
 \ref{proslspsl}
 and \ref{prososo}.

\begin{lemma}
\label{lemgh1h2}
Let $\g g$ be a semisimple Lie algebra, 
$\g k\subset \g g$ a semisimple Lie subalgebra
which is a direct sum $\g k=\g k_1\oplus\g k_2$
of two ideals of $\g k$, 
$\g h_1\subset\g k_1$ a semisimple Lie subalgebra
and $\g h:=\g h_1\oplus \g k_2$.
Assume that 
$\rho_{\g h_1}\leq \rho_{\g k_1/\g h_1}$
and $\rho_{\g k_2}\leq \rho_{\g g/\g k}$,
then one has $\rho_{\g h}\leq \rho_{\g q}$. 
\end{lemma}

Lemma \ref{lemgh1h2} is a special case of Lemma \ref{lemv1v2} (3).


\vspace{1em} 

Theorem \ref{thabcd} follows from Dynkin's classification of maximal semisimple Lie algebras in the classical Lie algebras by using the eight previous 
propositions.

\section{Exceptional simple Lie algebras}
\label{secexc}
In this chapter we give
a classification of the pairs $(\g g, \g h)$
of complex semisimple Lie algebras 
satisfying $\rho_{\g h} \not \le \rho_{\g g/\g h}$
in the case 
where $\g g$ is exceptional simple, 
and in particular,  
prove Theorems \ref{thrhrqc} and \ref{thghtempc}
for $\g g$ exceptional simple.

Throughout this chapter, 
$\g g$ is a complex exceptional simple Lie algebra,
$\g h$ is a complex semisimple Lie subalgebra
$\{0\} \neq \g h\neq \g g$, $\g q:=\g g/\g h$,
$\wt{\g h}$ the normalizer of $\g h$ in $\g g$,
$\g m$ is the generic stabilizer of $\g q$ in $\g h$,
and $\g m_s:=[\g m,\g m]$.

\subsection{Main list for exceptional Lie algebras}
\label{secmaiexc}

\begin{theorem}
\label{thefg}
Let $\g g=\g g_2$, $\g f_4$, $\g e_6$, $\g e_7$ or $\g e_8$
be a complex exceptional simple  Lie algebra. 
The complex semisimple Lie subalgebras 
$\g h\subsetneq \g g$ satisfying 
$\rho_\g h\not\leq\rho_\g q$ form the list in Table \ref{figefg}.
In this list
$\g q$ does not have $\rm{AGS}$ in $\g h$. 
\end{theorem}

\begin{figure}[htb]
\begin{center}
\begin{tabular}{|c|c|c|c|c|c|}
\hline
\!\!Case\!\!&$\g g$&maximal $\g h$&$\g m$&non maximal $\g h$&$\g m$\\
\hline
$G2$&$\g g_2$&$\g a_2$&$\g a_1\!\oplus\! \m C$&&\\
\hline
$F4$&$\g f_4$&$\g b_4$&$\g b_3$
&$\g d_4$&$\g a_2$\\
\hline
$E6.1$&$\g e_6$&$\g d_5$&$\g d_3\!\oplus\!\m C$&$\g b_4$&$\g a_1$\\
$E6.2$&$\g e_6$&$\g f_4$&$\g d_4$&$\g b_4$&$\g a_1$\\
\hline
$E7.1$&$\g e_7$&$\g d_6\!\oplus\!\g a_1$&$\g a_1\!\oplus\!\g a_1\!\oplus\!\g a_1$
&$\g d_6$&$\supset \g a_1$\\
$E7.2$&$\g e_7$&$\g e_6$&$\g d_4$&&\\
\hline
$E8$&$\g e_8$&$\g e_7\!\oplus\!\g a_1$&$\g d_4$
&$\g e_7$&$\g d_4$\\
\hline
\end{tabular}
\end{center}
\caption{
Pairs $(\g g,\g h)$ with $\rho_{\g h}\not\leq\rho_{\g q}$
for $\g g$ exceptional simple}
\label{figefg}
\end{figure}

Here are some comments on the pairs $(\g g,\g h)$ in this list with 
$\g h$ maximal.\\
The pair $(\g g_2,\g a_2)$ is the only one for which 
$(\g g,\wt{\g h})$ is not a symmetric pair.\\
The pair $(\g e_6,\g f_4)$ is  
the only one with ${\rm rank}\,\wt{\g h}<{\rm rank}\,\g g$.\\
The pairs 
$(\g f_4,\g b_4)$, $(\g e_7,\g a_1\!\oplus \!\g d_6)$ and
$(\g e_8,\g a_1\!\oplus \!\g e_{7})$ are equal rank symmetric pairs.\\
The pairs $(\g e_6,\wt{\g d}_5)$ and $(\g e_7,\wt{\g e}_{6})$, 
are equal rank Hermitian symmetric pairs.
\vspace{1em}

Once
we find the list of the pairs $(\g g, \g h)$ in Table \ref{figefg}, 
it is straightforward to verify
$\rho_{\g h} \not \le \rho_{\g q}$
for such $(\g g, \g h)$
by finding a witness (Definition \ref{def:witness}), 
or alternatively, 
by using Proposition \ref{prorhrv}
and checking 
that the generic stabilizer $\g m$ is nonabelian 
as indicated in Table \ref{figefg}.  
Thus the nontrivial part of Theorem \ref{thefg} is
 to prove
 that Table \ref{figefg} exhausts
 all the pairs $(\g g, \g h)$ satisfying $\rho_{\g h} \not \le \rho_{\g q}$.  
The strategy of this proof is to deal first with
pairs $(\g g,\g h)$ where $\g h$ is  maximal in $\g g$.
Dynkin's list of all these pairs is given in Section \ref{secdyn}.
These pairs are studied one by one in Section \ref{secpromax}
using upper bounds for the invariants $p_V$ which are stated in  
Section \ref{secirrsim} and explained in Chapter \ref{secbpvsim}.  
We find that there are exactly $7$ pairs $(\g g,\g h)$ 
with $\g g$ simple exceptional and $\g h$ semisimple maximal for which 
$\rho_{\g h}\not\leq \rho_{\g q}$~: they form the left-hand side of 
Table \ref{figefg}.

Then for each of these $7$  pairs $(\g g,\g h')$ we describe 
in Section \ref{secpronon}, the
semisimple Lie subalgebras $\g h\subset \g h'$ for which
one still has $\rho_{\g h}\not\leq \rho_{\g q}$. 
This uses also the bounds for the $p_V$s
 proven in Chapter \ref{secbpvsim}.

\subsection{Dynkin classification}
\label{secdyn}

For all complex simple Lie algebras $\g g$, 
Dynkin \cite{Dyn52} has classified maximal semisimple Lie subalgebras $\g h$.

\begin{figure}[htb]
\begin{center}
\begin{tabular}{|c|c|c|c|rcl|c|}
\hline
$\g g$&$\g h$&$\g q$&{\tiny i}&$d_{\g g}$&=&$d_{\g h}+d_{\g q}$&$\rm{AGS}$\\
\hline
$\g g_2$&$\g a_1\!\oplus\!\g a_1$&$S^3\m C^2\otimes\m C^2$&
{\tiny 2}&$14$&=&$6+8$&Y\\
&$\g a_2$&$\m C^3\oplus \text{dual}$&{\tiny 3}&$14$&=&$8+6$&N\\
\hline
$\g f_4$&$\g b_4$&$\m C^{16}$&{\tiny 2}&$52$&=&$36+16$&N\\
&$\g a_1\!\oplus\!\g c_3$&$\m C^2\otimes\La^3_0\m C^6$&{\tiny 2}&$52$&=&$24+28$&Y\\
&$\g a_2\!\oplus\!\g a_2$&$S^2\m C^3\otimes\m C^3\oplus \text{dual}$&{\tiny 3}&$52$&=&$16+36$&Y\\
\hline
$\g e_6$&$\g d_5$&$\m C\oplus(\m C^{16}\oplus \text{dual})$&{\tiny 1}&$78$&=&$45+33$&N\\
&$\g a_1\!\oplus\!\g a_5$&$\m C^2\otimes\La^3\m C^6$&{\tiny 2}&$78$&=&$38+40$&Y\\
&\!\!$\g a_2\!\oplus\!\g a_2\!\oplus\!\g a_2$\!\!&
$\scriptstyle\m C^3\otimes\m C^3\otimes\m C^3\oplus\; \text{dual}$&
{\tiny 3}&$78$&=&$24+54$&Y\\
\hline
$\g e_7$&$\g e_6$&$\m C\oplus(\m C^{27}\oplus \text{dual})$&{\tiny 1}&$133$&=&$78+55$&N\\
&$\g a_7$&$\La^4\m C^8$&{\tiny 2}&$133$&=&$63+70$&Y\\
&$\g a_1\!\oplus\!\g d_6$&$\m C^2\otimes\m C^{32}$&{\tiny 2}&$133$&=&$69+64$&N\\
&$\g a_2\!\oplus\!\g a_5$&$\m C^3\otimes\La^2\m C^6\oplus \text{dual}$&{\tiny 3}&$133$&=&$43+90$&Y\\
\hline
$\g e_8$&$\g d_8$&$\m C^{128}$&{\tiny 2}&$248$&=&$120+128$&Y\\
&$\g a_1\!\oplus\!\g e_7$&$\m C^2\otimes\m C^{56}$&{\tiny 2}&$248$&=&$136+112$&N\\
&$\g a_8$&$\La^3\m C^9\oplus \text{dual}$&{\tiny 3}&$248$&=&$80+168$&Y\\
&$\g a_2\!\oplus\!\g e_6$&$\m C^3\otimes\m C^{27}\oplus \text{dual}$
&{\tiny 3}&$248$&=&$86+162$&Y\\
&$\g a_4\!\oplus\!\g a_4$&
\!\!$\scriptstyle (\La^2\m C^5\otimes\m C^5
\oplus\;\m C^5\otimes\La^2\m C^5)\oplus\; \text{dual}$\!\!&
{\tiny 5}&$248$&=&$48+200$&Y\\
\hline
\end{tabular}
\end{center}
\caption{$R$-subalgebras of exceptional simple Lie algebras}
\label{figral}
\end{figure}

\begin{definition}
\label{defrssub} 
A maximal semisimple Lie subalgebra $\g h$ of $\g g$ is called an {\it{$R$-subalgebra}} 
if ${\rm rank}\,\wt{\g h}={\rm rank}\,\g g$ and called an {\it{$S$-subalgebra}}
if  ${\rm rank}\,\wt{\g h}<{\rm rank}\,\g g$.
\end{definition}

The classification of $R$-subalgebras goes back to Borel--Siebenthal. Later,
Dynkin has given a nice interpretation of this list using the so called
extended Dynkin diagram. This list is given in Table \ref{figral}.

\begin{figure}[htb]
\begin{center}
\begin{tabular}{|c|c|c|rcl|c|}
\hline
$\g g$&$\g h$&$\g q$&$d_{\g g}$&=&$d_{\g h}+d_{\g q}$&$\rm{AGS}$\\
\hline
$\g g_2$&$\g a_1$&$*$&$14$&=&$3+11$&Y\\
\hline
$\g f_4$&$\g a_1$&$*$&$52$&=&$3+49$&Y\\
&$\g a_1\!\oplus\!\g g_2$&
$*$
&$52$&=&$17+35$&Y\\
\hline
$\g e_6$&$\g a_2$&$*$&$78$&=&$8+70$&Y\\
&$\g g_2$&$*$&$78$&=&$14+64$&Y\\
&$\g c_4$&$\La^4_0\m C^8$&$78$&=&$36+42$&Y\\
&$\g f_4$&$\m C^{26}$&$78$&=&$52+26$&N\\
&$\g a_2\!\oplus\!\g g_2$&$\m C^8\otimes\m C^7$&$78$&=&$22+56$&Y\\
\hline
$\g e_7$&$\g a_1\; \mbox{\tiny(twice)}$&$*$&$133$&=&$3+130$&Y\\
&$\g a_2$&$*$&$133$&=&$8+125$&Y\\
&$\g a_1\!\oplus\!\g a_1$&$*$&$133$&=&$6+127$&Y\\
&$\g a_1\!\oplus\!\g g_2$&$*$&$133$&=&$17+116$&Y\\
&$\g a_1\!\oplus\!\g f_4$&$\m C^3\otimes\m C^{26}$&$133$&=&$55+78$&Y\\
&$\g g_2\!\oplus\!\g c_3$&$\m C^7\otimes\La^2_0\m C^6$&$133$&=&$35+98$&Y\\
\hline
$\g e_8$&$\g a_1 \;\mbox{\tiny(3 times)}$&$*$
&$248$&=&$3+245$&Y\\
&$\g b_2$&$*$&
$248$&=&$10+238$&Y\\
&$\g a_1\!\oplus\!\g a_2$& $*$ &
$248$&=&$11+237$&Y\\
&\!\!$\g a_1\!\oplus\!\g g_2\!\oplus\!\g g_2$\!\!&
\!\!$\scriptstyle\m C^3\otimes\m C^7\otimes\m C^7
\oplus\;\m C^5\otimes(\m C^7\otimes\m C
\oplus\; \m C\otimes\m C^7)$\!\!&
$248$&=&$31+217$&Y\\
&$\g g_2\!\oplus\!\g f_4$&$\m C^7\otimes\m C^{26}$&
$248$&=&$66+182$&Y\\
\hline
\end{tabular}
\end{center}
\caption{$S$-subalgebras of exceptional simple Lie algebras}
\label{figsal}
\end{figure}

The classification of $S$-subalgebras is due to Dynkin
(except for $\g a_1\!\oplus\!\g g_2\!\oplus\!\g g_2$ in $\g e_8$
 which is forgotten there, 
see \cite{LiSe04}).
The list is given in Table \ref{figsal}.

Here are a few comments on Tables \ref{figral} and \ref{figsal}.
To each $R$-subalgebra, Dynkin associates an integer 
$i=1$, $2$, $3$ or $5$.  When $i=1$, $\wt{\g h}$ has a one-dimensional center
and $(\g g,\wt{\g h})$ is the complexification of a Hermitian symmetric pair.
When $i\geq 2$, one has $\wt{\g h}=\g h$ and $\g h$ is the set of fixed points 
of an automorphism of $\g g$ of order $i$.

In Tables \ref{figral} and \ref{figsal}, 
the third column describes $\g q$ as a 
representation of $\g h$. For Table \ref{figral}, 
it is obtained thanks to Dynkin's construction
of the $R$-subalgebras using the extended Dynkin diagrams.
For Table \ref{figsal}, 
it is based on Tables 10.1 in \cite[p.~214--215]{LiSe04}.

In this third column, the notation $\m C^n$ means ``one of irreducible
represen\-tation of dimension $n$''. 
The slight ambiguity does not affect our consequence. 
For instance 
$\m C^{27}$ is one of the two $27$-dimensional 
irreducible represen\-tations of $\g e_6$, 
 which are dual to each other. 
Similarly 
$\m C^{16}$ is one of the two dual $16$-dimensional 
irreducible representations of $\g d_5$ called the half-spin representation.
As a representation of $\g b_4$, $\m C^{16}$ 
is still irreducible and is orthogonal.
The subscript $0$ in notations like $\La_0^3\m C^6$ means 
the irreducible subrepresentation spanned by the highest weight vector
{\it{e.g.}}, 
 $\La_0^3 \m C^6 \simeq \La^3\m C^6/\m C^6$.

In Tables \ref{figral} and \ref{figsal}, the last column tells us 
({\it{Yes}} or {\it{No}})
 according to whether $\g h$ has $\rm{AGS}$ in $\g q$ or not. 
The answers {\it{No}} are deduced from the fact that those pairs 
$(\g g,\g h)$ are equal to the complexification 
$(\g g_{1,\m C},\g k_{1,\m C})$ of a Riemannian symmetric pair
$(\g g_1,\g k_1)$ for which the real semisimple Lie algebra 
$\g g_1$ is not quasisplit
(except for  the pair $(\g g_2,\g a_2)$ for which one computes directly 
that the generic stabilizer is $\g m=\g a_1$).
The answers {\it{Yes}} will be deduced from Proposition \ref{prorhrv},
once we will have checked the inequality $\rho_{\g h}\leq \rho_{\g q}$.
\vspace{1em}

\subsection{Irreducible representations of simple Lie algebras}
\label{secirrsim}

In order to prove Theorem \ref{thefg}, we will need to 
compute accurately the real number $p_V$ 
defined in \eqref{eqnpvinf} 
for many irreducible representations $V$
of simple Lie algebras $\g h$. 
Most of the results that we will need are contained in 
Tables \ref{figair}, \ref{fignir}, and \ref{figeir} below.

\begin{figure}[htb]
\begin{center}
\begin{tabular}{|c|c|c|c|c|c|}
\hline
$\g h$&$V$&$p_V$&parameter&duality&name\\
\hline
$\g a_1$&
$\m C^{2}$&
$2$&&sympl.&standard\\
$\g a_3$&
$\La^2\m C^{4}$&
$4$&&orth.&$V_{\om_2}$\\
$\g a_5$&
$\La^3\m C^{6}$&
$2$&&sympl.&$V_{\om_3}$\\
\hline
$\g b_\ell$&
$\m C^{2\ell+1}$&
$2\,\ell\!-\! 1$&
$\ell\geq 2$&
orth.&
standard\\
$\g b_2$&
$\m C^{4}$&
$4$&&
sympl.&
spin
\\
$\g b_3$&
$\m C^{8}$&
$4$&&
orth.&
spin
\\
$\g b_4$&
$\m C^{16}$&
$3$&&
orth.&
spin
\\
$\g b_5$&
$\m C^{32}$&
$2$&&
sympl.&
spin
\\
$\g b_6$&
$\m C^{64}$&
$4/3$&&
sympl.&
spin
\\
\hline
$\g c_\ell$&
$\m C^{2\,\ell}$&
$2\,\ell$&
$\ell\geq 3$& 
sympl.&
standard\\
&
$\La^2_0\m C^{2\,\ell}$&
$\tfrac{\ell+1}{\ell-1}$&
$\ell\geq 3$&
orth.&
$V_{\om_2}$\\
$\g c_3$&
$\La^3_0\m C^{6}$&
$5/3$&&
sympl.&
$V_{\om_3}$\\
\hline
$\g d_\ell$&
$\m C^{2\,\ell}$&
$2\,\ell\! -\! 2$&
$\ell\geq 4$& 
orth.&
standard\\
$\g d_4$&
$\m C^{8}$&
$6$&& 
orth.&
half-spin\\
$\g d_6$&
$\m C^{32}$&
$5/2$&& 
sympl.&
half-spin\\
\hline
$\g e_7$&
$\m C^{56}$&
$17/6$&&
sympl.& 
$V_{\om_7}$\\
\hline
$\g f_4$&
$\m C^{26}$&
$8/3$&&
orth.& 
$V_{\om_4}$\\
\hline
$\g g_2$&
$\m C^{7}$&
$3$&&
orth.& 
$V_{\om_1}$\\
\hline
\end{tabular}
\end{center}
\caption{Self-dual irreducible faithful representations $V$
 of simple Lie algebra $\g h$ with 
$p_V\!>\!1$}
\label{figair}
\end{figure}

\begin{theorem}
\label{thirrsim}
Let $\g h$ be a complex simple Lie algebra
 and $V$ be an irreducible faithful representation of $\g h$.\\
If $V$ is self-dual  and $\rho_\g h\not\leq \rho_V$, then 
$(\g h,V)$ is in Table \ref{figair}.\\
If $V$ is not self-dual  and $\rho_\g h\not\leq 2\, \rho_V$, then 
$(\g h,V)$ or $(\g h,V^*)$ is in Table \ref{fignir}.
\end{theorem}

\begin{figure}[htb]
\begin{center}
\begin{tabular}{|c|c|c|c|c|}
\hline
$\g h$&$V$&$p_V$&parameter&name
\\
\hline
$\g a_\ell$&
$\m C^{\ell+1}$&
$2 \ell$&
$\ell\geq 2$&
standard\\
$\g a_\ell$&
$\La^2\m C^{\ell+1}$&
$2\,\tfrac{\ell+2}{\ell}$&
$\ell\geq 4$, even&
$V_{\om_2}$\\
&&
$2\,\tfrac{\ell+1}{\ell-1}$&
$\ell\geq 5$, odd&
$V_{\om_2}$\\
\hline
$\g d_5$&
$\m C^{16}$&
$7/2$&& 
half-spin\\
\hline
$\g e_6$&
$\m C^{27}$&
$7/2$&& 
$V_{\om_1}$\\
\hline
\end{tabular}
\end{center}
\caption{Non-self-dual irreducible representations for $\g h$ simple with 
$p_V\!>\! 2$}
\label{fignir}
\end{figure}


Theorem \ref{thirrsim} will be explained in Chapter \ref{secbpvsim}. 
The following corollary tells us that Conjecture \ref{conrhrv} 
is true when $\g h$ is simple.

\begin{cor}
\label{corirrsim}
Let $\g h$ be a complex simple Lie algebra
and $V$ a complex orthogonal representation of  $\g h$. One has the equivalence:
\begin{eqnarray}
\mbox{$V$ has $\rm{AGS}$ in $\g h$}
&\Longleftrightarrow &
\rho_{\g h} \leq \rho_{V}.
\end{eqnarray}
\end{cor}

We already know from Proposition \ref{prorhrv}
 the implication $\Leftarrow$
 in Corollary \ref{corirrsim} holds.  
The opposite implication of 
Corollary \ref{corirrsim} is proven by decomposing  $V$ 
into irreducible components 
and by checking, using Tables \ref{figair} and  \ref{fignir}, that when $\rho_\g h\not\leq \rho_V$
then the generic stabilizer is not abelian.

\subsection{Checking Theorem \ref{thefg} for $\g h$  maximal}
\label{secpromax}

We just have to check that all pairs $(\g g,\g h)$ occurring 
in Dynkin's classification (Tables \ref{figral} and \ref{figsal})
 with \lq\lq{$Y$}\rq\rq\ in the last column satisfy $\rho_\g h\leq \rho_\g q$.

For the $12$ cases with $\g h$  simple of rank $1$ or $2$, this follows from 
Corollary \ref{corbpvaabg}. We just notice that when $\g h$ is an $S$-subalgebra,
the centralizer of $\g h$ is trivial.

For the $5$ cases with $\g h$  product of $\g a_1$ by a simple 
Lie algebra $\g h_2$ of rank $1$ or $2$, this follows from 
Corollary \ref{corbpvaaabg}. 
We just notice that when $\g h'$ is a
non-zero ideal of an $S$-subalgebra,
the centralizer of $\g h'$ is included in $\g h$.

For the $4$ cases with $\g h$ simple of rank  $\geq 3$, this follows from 
Table \ref{figeir} which is part of Theorem \ref{thirrsim}. 

\begin{figure}[htb]
\begin{center}
\begin{tabular}{|c|c|c|c|c|}
\hline
$\g h$&$V$&$p_V$&duality&name\\
\hline
$\g a_7$&
$\La^4\m C^{8}$&
$p_V\leq 1$&orth.&$V_{\om_4}$\\
$\g a_8$&
$\La^3\m C^{9}$&
$p_V\leq 2$&non-auto.&$V_{\om_3}$\\
$\g c_4$&
$\La^4_0\m C^{8}$&
$p_V\leq 1$&
orth.&
$V_{\om_4}$\\
$\g d_8$&
$\m C^{128}$&
$p_V\leq 1$&
orth.&
half-spin\\
\hline
\end{tabular}
\end{center}
\caption{``Useful'' representations 
which are not in Tables \ref{figair} and \ref{fignir}}
\label{figeir}
\end{figure}

For the  $7$ remaining cases in Table \ref{figral}, we conclude with  Lemma
\ref{lemh1h2r}.

For the  $5$ remaining cases in Table \ref{figsal}, we conclude with Lemma \ref{lemh1h2s}. 

\begin{lemma}
\label{lemh1h2r}
$(1)$ Let $\g h=\g a_1\!\oplus\!\g c_3$ and $V=\m C^2\otimes\La^3_0\m C^6$.
Then one has $\rho_\g h\leq \rho_V$.\\
$(2)$ Let $\g h=\g a_2\!\oplus\!\g a_2$ and 
$V=S^2\m C^3 \otimes\m C^3$.
Then one has $\rho_\g h\leq 2\rho_V$.\\
$(3)$ Let $\g h=\g a_2\!\oplus\!\g a_2\!\oplus\!\g a_2$ and 
$V=\m C^3 \otimes\m C^3\otimes\m C^3$.
Then one has $\rho_\g h\leq 2\rho_V$.\\
$(4)$ Let $\g h=\g a_1\!\oplus\!\g a_5$ and $V=\m C^2\otimes\La^3\m C^6$.
Then one has $\rho_\g h\leq \rho_V$.\\
$(5)$ Let $\g h=\g a_2\!\oplus\!\g a_5$ and $V=\m C^3\otimes\La^2\m C^6$.
Then one has $\rho_\g h\leq 2\rho_V$.\\
$(6)$ Let $\g h=\g a_2\!\oplus\!\g e_6$ and $V=\m C^3\otimes\m C^{27}$.
Then one has $\rho_\g h\leq 2\rho_V$.\\
$(7)$ Let $\g h=\g a_4\!\oplus\!\g a_4$ and $V=\m C^5\otimes\La^2\m C^5$.
Then one has $\rho_\g h\leq 4\rho_V$.
\end{lemma}

\begin{lemma}
\label{lemh1h2s}
$(1)$ Let $\g h=\g a_2\!\oplus\!\g g_2$ and $V=\m C^8\otimes\m C^7$.
Then one has $\rho_\g h\leq \rho_V$.\\
$(2)$ Let $\g h=\g a_1\!\oplus\!\g f_4$ and $V=S^2\m C^2\otimes\m C^{26}$.
Then one has $\rho_\g h\leq \rho_V$.\\
$(3)$ Let $\g h=\g g_2\!\oplus\!\g c_3$ and $V=\m C^7\otimes\La^2_0\m C^6$.  
Then one has $\rho_\g h\leq \rho_V$.\\
$(4)$ Let $\g h=\g a_1\!\oplus\!\g g_2\!\oplus\!\g g_2$ and 
$V=\m C^3\otimes\m C^7\otimes\m C^7$.
Then one has $\rho_\g h\leq \rho_V$.\\
$(5)$ Let $\g h=\g g_2\!\oplus\!\g f_4$ and $V=\m C^7\otimes\m C^{26}$.
Then one has $\rho_\g h\leq \rho_V$.
\end{lemma}

\begin{proof}[Checking Lemma \ref{lemh1h2r}]
We will write $\g h=\g h_1\oplus \g h_2$ 
with $\g h_1$ simple and $V=V_1\otimes V_2$.
For $i=1$, $2$, we will write $p_i=p_{V_i}$, $d_i=\dim V_i$
and apply the bound $p_V\leq \tfrac{p_1}{d_2}+\tfrac{p_2}{d_1}$ from 
Lemma \ref{lemv1v2}, and the values of $p_i$ given 
in Tables \ref{figair} and \ref{fignir}.\\
$(1)$ One has $d_1=2$, $p_1=2$, $d_2=14$, $p_2=\tfrac{5}{3}$, 
hence
$p_V\leq \tfrac{2}{14}+\tfrac{5}{6}\leq 1$.\\
$(2)$ One has $d_1=6$, $p_1=\tfrac{4}{3}$, $d_2=3$, $p_2=4$, 
hence
$p_V\leq \tfrac{4}{9}+\tfrac{4}{6}\leq 2$.\\
$(3)$ One has $d_1=3$, $p_1=4$, $d_2=9$, $p_2\le\tfrac{8}{3}$, 
hence
$p_V\leq \tfrac{4}{9}+\tfrac{8}{9}\leq 2$.\\
$(4)$ One has $d_1=2$, $p_1=2$, $d_2=20$, $p_2=2$. This is not enough to conclude.  
But a direct computation shows $\rho_{\g h} \le \rho_V$. 
\\
$(5)$ One has $d_1=3$, $p_1=4$, $d_2=15$, $p_2=3$, 
hence
$p_V\leq \tfrac{4}{15}+\tfrac{3}{3}\leq 2$.\\
$(6)$ One has $d_1=3$, $p_1=4$, $d_2=27$, $p_2=\tfrac{7}{2}$, 
hence
$p_V\leq \tfrac{4}{27}+\tfrac{7}{6}\leq 2$.\\
$(7)$ One has $d_1=5$, $p_1=8$, $d_2=10$, $p_2=3$, 
 hence
$p_V\leq \tfrac{8}{10}+\tfrac{3}{5}\leq 4$. 
\end{proof}

\begin{proof}[Checking Lemma \ref{lemh1h2s}]
We use the same notations.\\
$(1)$ One has $d_1=8$, $p_1=1$, $d_2=\;\, 7$, $p_2=3$, 
 hence
$p_V\leq \tfrac{1}{7}+\tfrac{3}{8}\leq 1$.\\ 
$(2)$ One has $d_1=3$, $p_1=1$, $d_2=26$, $p_2=\tfrac{8}{3}$, 
 hence
$p_V\leq \tfrac{1}{26}+\tfrac{8}{9}\leq 1$.\\
$(3)$ One has $d_1=7$, $p_1=3$, $d_2=14$, $p_2=2$, 
 hence
$p_V\leq \tfrac{3}{14}+\tfrac{2}{7}\leq 1$.  
\\
$(4)$ One has $d_1=3$, $p_1=1$, $d_2=49$, $p_2\le\tfrac{6}{7}$,
hence
$p_V\leq \tfrac{1}{49}+\tfrac{6}{21}\leq 1$.\\ 
$(5)$ One has $d_1=7$, $p_1=3$, $d_2=26$, $p_2=\tfrac{8}{3}$, 
 hence
$p_V\leq \tfrac{3}{26}+\tfrac{8}{21}\leq 1$. 
\end{proof}

\subsection{Checking Theorem \ref{thefg} for $\g h$  non-maximal}
\label{secpronon}


We first consider the ca\-se where $\g h$ is maximal in a maximal
semisimple Lie algebra $\g h'$ of $\g g$.
Taking an $\g h'$-invariant subspace $\g q'$ in $\g g$
 and an $\g h$-invariant subspace $\g q''$ in $\g h'$, 
 we write 
$$
\g g=\g h'\oplus\g q'
\; ,\;\; 
\g h'=\g h\oplus \g q''
\;\;{\rm and}\;\;
\g q=\g q'\oplus\g q''
$$
According to Section \ref{secpromax}, 
the pair $(\g g,\g h')$ is among the $7$ pairs 
in the left side of Table \ref{figefg}. 
Moreover the Lie algebra $\g h$ is also a maximal 
semisimple Lie subalgebra of $\g h'$ satisfying 
$\rho_{\g h}\not\leq \rho_{\g q''}$.
One can find the list of such subalgebras $\g h$
from Theorem \ref{thabcd} and Table \ref{figefg}.
Hence the triple $(\g g,\g h',\g h)$ has to be in the
following  Table \ref{figghh}.

Note that triples $(\g g,\g h',\g h)$ like 
$(\g f_4,\g b_4,\g b_3)$ or $(\g e_6,\g d_5,\g d_4)$ 
do not occur in Table \ref{figghh} because in these examples
$\g h$ is not maximal in $\g h'$. 
Such examples will be taken care of in Lemma \ref{lemmamama}. 
Similarly, the triple $(\g g_2,\g a_2,\g a_1)$ does not occur
in Table \ref{figghh} because in this example $\rho_{\g h}\leq \rho_{\g q''}$.

\begin{figure}[htb]
\begin{center}
\begin{tabular}{|c|c|c|c|c|l|c|}
\hline
$\g g$&$\g h'$&$\g h$&$\g q' \simeq \g g/\g h'$
&$\g q''\simeq \g h'/\g h$&
\!\!\mbox{\small dimension}\!\! 
&
\!\!$\mbox{\small\rm AGS}$\!\!
\\
&&&&&
$\g g \!=\! \g h\! +\! \g q$
&
\\
\hline
$\g f_4$&$\g b_4$&$\g d_4$&
$\m C^8\oplus\m C^8$&
$\m C^8$&
$\scriptstyle52= 28+24$&N\\
&&$\g b_1\g d_3$&
$(\m C^2\!\otimes\m C^4)\oplus \text{dual}$&
$\m C^3\otimes\m C^6$&
$\scriptstyle52= 18+34$&Y
\\
\hline
$\g e_6$&$\g d_5$&$\g b_4$&
$\m C\oplus(\m C^{16}\oplus\m C^{16})$&
$\m C^9$&
$\scriptstyle 78=36+42$&N\\
&&$\g b_1\g b_3$&
$\m C\oplus(\m C^2\!\otimes\m C^8\oplus \text{dual})$&
$\m C^3\otimes\m C^7$&
$\scriptstyle78=24+54$&Y
\\
&&$\g a_4$&
$\scriptstyle\m C\oplus((\m C \oplus \m C^5\oplus\La^2\m C^5)\oplus \text{dual})$&
$\scriptstyle\m C\oplus(\La^2\m C^5\oplus \text{dual})$&
$\scriptstyle 78=24+54$&Y
\\
&$\g f_4$&$\g b_4$&
$\m C\oplus\m C^{9}\oplus\m C^{16}$&
$\m C^{16}$&
$\scriptstyle78=36+42$&N
\\
\hline
$\g e_7$&$\g e_6$&$\g f_4$&
$\m C\!\oplus\!\m C\!\oplus\!\m C\!\oplus\!\m C^{26}\!\oplus\!\m C^{26}$&
$\m C^{26}$&
$\scriptstyle133=45+88$&Y
\\
&&$\g d_5$&
$\scriptstyle\m C\oplus((\m C\oplus\m C^{10}\oplus\m C^{16})\oplus \text{dual})$&
$\m C\!\oplus\!(\m C^{16}\!\oplus\! \text{dual})$&
$\scriptstyle133=52+81$&Y
\\
&$\!\!\g a_1\!\g d_6\!\!$&\!\!$\g d_6$\!\!&
$\m C^{32}\oplus\m C^{32}$&
$\m C\!\oplus\!\m C\!\oplus\!\m C$&
$\scriptstyle133=66+67$&N
\\
&&$\g a_1\g b_5$&
$\m C^2\otimes\m C^{32}$&
$\m C\otimes\m C^{11}$&
$\scriptstyle133=58+75$&Y
\\
&&\!\!$\g a_1\g b_1\g b_4$\!\!&
$\m C^2\otimes\m C^2\otimes\m C^{16}$&
$\m C\!\otimes\!\m C^3\!\otimes\!\m C^{9}$&
$\scriptstyle133=42+91$&Y
\\
&&\!\!$\g a_1\g d_2\g d_4$\!\!&
\!\!$\m C^2\!\otimes\!(\m C^2\!\otimes\!\m C^8\!\oplus\!\m C^2 \!\otimes\! \m C^8)$\!\!&
$\m C\!\otimes\!\m C^4\!\otimes\!\m C^{8}$&
$\scriptstyle 133=37+96$&Y
\\
&&$\g a_1\g a_5$&
\!\!$\m C^2\!\otimes\!((\m C\!\oplus\!\La^2\m C^6)\!\oplus\! \text{dual})$\!\!&
\!\!$\scriptstyle \m C\oplus ((\m C\otimes\La^2\m C^6)\oplus \text{dual})$\!\!&
$\scriptstyle 133=38+95$&Y
\\
&&$\g a_1\g a_5$&
\!\!$\m C^2\!\otimes\!(\m C^6\!\oplus\!\m C^6 \!\oplus\! \La^3 \m C^6)$\!\!&
\!\!$\scriptstyle \m C\oplus ((\m C\otimes\La^2\m C^6)\oplus \text{dual})$\!\!&
$\scriptstyle 133=38+95$&Y
\\
\hline
$\g e_8$&\!\!$\!\g a_1\!\g e_7$\!\!&$\g e_7$&
$\m C^{56}\oplus\m C^{56}$&
$\m C\!\oplus\!\m C\!\oplus\!\m C$&
\!\!$\scriptstyle 248=133+115$\!\!&N
\\
&&$\g a_1\g e_6$&
\!\!$\m C^2\!\otimes\!((\m C\!\oplus\!\m C^{27} )\!\oplus\! \text{dual})$\!\!&
\!\!$\m C\!\oplus\! (\m C^{27}\!\oplus\! \text{dual})$\!\!&
\!$\scriptstyle 248=81+167$\!&Y
\\
&&\!\!$\g a_1\!\g a_1\!\g d_6$\!\!&
\!\!$\m C^2\!\otimes\!(\m C\!\otimes\!\m C^{32}\!\oplus\!\m C^2\!\otimes\!\m C^{12})$\!\!&
\!\!$\m C\otimes\m C^2\otimes \m C^{32}$\!\!&
\!$\scriptstyle 248=72+176$\!&Y\\
\hline
\end{tabular}
\end{center}
\caption{Triples $(\g g,\g h',\g h)$ to be studied}
\label{figghh}
\end{figure}

In this Table \ref{figghh}
we describe $\g q'$ and $\g q''$
 as a representation of  $\g h$,
using Tables \ref{figral} and \ref{figsal}, which describes $\g q'$ as a
representation of $\g h'$
and decomposing this representation as a sum of irreducible 
representations of $\g h$, 
{\it{i.e.}}, 
the branching law for $\g h' \downarrow \g h$.  
There are two realizations of $\g a_5= \mathfrak{s l}_6$
 in $\g d_6 = \mathfrak{s o}_{12}$, 
 which is conjugate by an outer automorphism 
 of $\g d_6$, 
 but this automorphism does not extend to $\g g=\g e_7$.  
Accordingly, 
 we have needed to list two different structures
 of $\g q'= \g g/\g h'$
 as $\g h$-modules for the triple
 $(\g g, \g h', \g h)=(\g e_7, \g a_1 \g d_6, \g a_1 \g a_5)$.  
The Lie algebra $\g d_4=\g s\g o_8$ has three $8$-dimensional 
irreducible representations $V_{\om_1}$, $V_{\om_3}$, $V_{\om_4}$, 
we have noted all of them 
as $\m C^8$ since we will not need to know which is which.

We already know that for all triples $(\g g,\g h',\g h)$ occurring 
in this Table \ref{figghh}
with \lq\lq{N}\rq\rq\ in the last column
 where $\g q$ does not have $\rm{AGS}$ in $\g h$, 
 hence,
by Proposition \ref{prorhrv}, they satisfy $\rho_\g h\not\leq \rho_\g q$.

It remains to check that all triples $(\g g,\g h',\g h)$ occurring 
in this Table \ref{figghh}
with \lq\lq{$Y$}\rq\rq\
 in the last column satisfy $\rho_\g h\leq \rho_\g q$.

For the $3$
cases where $\g h$ is simple,
this follows directly from Inequality \eqref{eqnpvpvpv} 
and Lemma \ref{lemghh1}. 

For the $9$
cases where $\g h$ is not simple,
this follows directly from Inequality \eqref{eqnpvpvpv} 
and Lemma \ref{lemghh2}.

\begin{lemma}
\label{lemghh1}
$(1)$
Let $\g h=\g a_4$ and $V=\La^2\m C^5$.
Then one has $p_V\leq 3$.
\\
$(2)$ Let $\g h=\g f_4$ and $V=\m C^{26}$.
Then one has $p_V\leq 3$.\\
$(3)$ Let $\g h=\g d_5$ and $V=\m C^{16}$.
Then one has $p_V\leq 4$.
\end{lemma}


\begin{proof}[Checking Lemma \ref{lemghh1}]
These values are obtained from Tables \ref{figair} and \ref{fignir}.
\end{proof}

\begin{lemma}
\label{lemghh2}
$(1)$ Let $\g h=\g b_1\!\oplus\!\g d_3$, $V'=\m C^2\otimes\m C^4$, and
$V''=\m C^3\otimes\m C^{6}$.\\
Then one has $p_{V'}\leq 4$ and $p_{V''}\leq 2$.\\
$(2)$ Let $\g h=\g b_1\!\oplus\!\g b_3$, $V'=\m C^2\otimes\m C^8$
and $V''=\m C^3\otimes\m C^{7}$.\\
Then one has $p_{V'}\leq 4$ and $p_{V''}\leq 2$.\\
$(3)$ Let $\g h=\g a_1\!\oplus\!\g b_5$ and $V=\m C^2\otimes\m C^{32}\oplus
\m C\otimes\m C^{11}$.
Then one has $p_V\leq 1$.
\\
$(4)$
 Let $\g h=\g a_1\!\oplus\!\g b_1\!\oplus\!\g b_4$ and 
$V=\m C^2\otimes\m C^2\otimes\m C^{16}$.
Then one has $p_V\leq 1$.
\\
$(5)$ Let $\g h\!=\!\g a_1\!\oplus\!\g d_2\!\oplus\!\g d_4$ and 
$V\!=\!\m C^2\!\otimes\!\m C^2\!\otimes\!\m C^8\oplus
\m C\!\otimes\!\m C^4\!\otimes\!\m C^{8}$.
Then one has $p_V\leq 1$.\\
$(6)$ Let $\g h=\g a_1\!\oplus\!\g a_5$ and 
$V=\m C^2\otimes\La^2\m C^{6}$.
Then one has $p_V\leq 2$.
\\
$(7)$ Let $\g h=\g a_1\!\oplus\!\g a_5$ and 
$V=\m C^2\otimes (\m C^6 \oplus \La^3 \m C^6 \oplus \m C^6)$.  
Then one has $p_V\leq 1$.  
\\
$(8)$ Let $\g h=\g a_1\!\oplus\!\g e_6$ and 
$V=\m C^2\otimes\m C^{27}$.
Then one has $p_V\leq 2$.\\
$(9)$ Let $\g h=\g a_1\!\oplus\!\g a_1\!\oplus\!\g d_6$ and 
$V\!=\!(\m C^2\otimes\m C\oplus\m C\otimes\m C^2)\otimes\m C^{32}$.
Then one has $p_V\leq 1$.
\end{lemma}

\begin{proof}[Checking Lemma \ref{lemghh2}]
In all proofs, we will write $\g h=\g h_1\oplus \g h_2$ 
with $\g h_1$ simple and $V=V_1\otimes V_2$.  
For $i=1$, $2$, we will write $p_i=p_{V_i}$, $d_i=\dim V_i$
and apply the bound $p_V\leq \tfrac{p_1}{d_2}+\tfrac{p_2}{d_1}$ from 
Lemma \ref{lemv1v2} (2),
 and the values of $p_i$ given 
in Tables \ref{figair} and \ref{fignir}. And similarly with primes and double primes.\\
$(1)$ One has $d'_1=2$, $p'_1=2$, $d'_2=4$, $p'_2=6$.  
Hence
$p_{V'}\leq \tfrac{2}{4}+\tfrac{6}{2}\leq 4$.\\
One has $d''_1=3$, $p''_1=1$, $d''_2=6$, $p''_2=4$.  
Hence
$p_{V''}\leq \tfrac{1}{6}+\tfrac{4}{3}\leq 2$.\\
$(2)$ One has $d'_1=2$, $p'_1=2$, $d'_2=8$, $p'_2=4$.  
Hence
$p_{V'}\leq \tfrac{2}{8}+\tfrac{4}{2}\leq 4$.\\
One has $d''_1=3$, $p''_1=1$, $d''_2=7$, $p''_2=5$.  
Hence
$p_{V''}\leq \tfrac{1}{7}+\tfrac{5}{3}\leq 2$.\\
$(3)$ We write $V=V'\oplus V''$. 
One has $d'_1=2$, $p'_1=2$, $d'_2=32$, $p'_2=2$. 
This is not enough to conclude.
But a direct computation shows
 $p_{V'}=1$
 for the irreducible $\g h$-module $V'=\m C^2 \otimes \m C^{36}$.  
$(4)$ One has $d_1=4$, $p_1=2$, $d_2=16$, $p_2=3$ and hence
$p_V\leq \tfrac{2}{16}+\tfrac{3}{4}\leq 1$.\\
$(5)$ We first check as above that if 
$\g h_0=\g a_1\!\oplus\!\g a_1\!\oplus\!\g d_4$ and 
$W=\m C^2\otimes\m C^2\otimes\m C^8$, then $\rho_{\g h_0}\leq 2\rho_W$.
Now, one has $\g h=\g h_1\!\oplus\! \g h_2\!\oplus\! \g h_3\!\oplus\! \g h_4$
with $\g h_1=\g h_2=\g h_3=\g a_1$ and $\g h_4=\g d_4$ and
$V$ is the sum of three irreducible components $V'\oplus V''\oplus V'''$
with kernel, respectively $\g h_1$, $\g h_2$ and $\g h_3$. 
Hence one has the bound
$\rho_{\g h_2\oplus\g h_3\oplus \g h_4}\leq 2\rho_{V'}$,
and similarly for $V''$ and $V'''$. Adding these three inequalities gives
$\rho_{\g h}\leq \rho_V$.\\
$(6)$ One has $d_1=2$, $p_1=2$, $d_2=15$, $p_2=3$.  
Hence
$p_V\leq \tfrac{2}{15}+\tfrac{3}{2}\leq 2$.
\\
$(7)$ We write $V=V' \oplus V'' \oplus V'''$.  
One has $d_1'=2$, $p_1'=2$, $d_2'=6$, $p_2'=10$, 
 hence 
 $p_{V'} \le \tfrac 2 6 + \tfrac{10}2 = \tfrac{16}{3}$. 
One has $d_1''=2$, $p_1''=2$, $d_2''=20$, $p_2''=2$, 
 hence 
 $p_{V''} \le \tfrac 2 {20} + \tfrac{2}2 = \tfrac{11}{10}$. 
Thus
$p_{V}^{-1} \ge p_{V'}^{-1}+p_{V''}^{-1}+p_{V'''}^{-1} = \tfrac 3 {16} + \tfrac{10}{11}+ \tfrac{3}{16} \ge 1$. 
\\ 
$(8)$ One has $d_1=2$, $p_1=2$, $d_2=27$, $p_2=\tfrac{7}{2}$. 
Hence
$p_V\leq \tfrac{2}{27}+\tfrac{7}{4}\leq 2$.\\ 
$(9)$ One has $d_1=4$, $p_1=2$, $d_2=32$, $p_2=\tfrac{5}{2}$.  
Hence
$p_V\leq \tfrac{2}{32}+\tfrac{5}{8}\leq 1$.
\end{proof}

\begin{proof}[Ending the proof of Theorem \ref{thefg}]
The following Lemma \ref{lemmamama} tells us that we have already
encountered all possible cases.
\end{proof}

\begin{lemma}
\label{lemmamama}
Let $\g g$ be a simple exceptional complex Lie algebra
and $\g h\subsetneq\g g$ a semisimple Lie subalgebra such that 
$\rho_{\g h}\not\leq\rho_{\g q}$.
Then either $\g h$ is maximal  in $\g g$,
or $\g h$ is maximal in a maximal
semisimple Lie algebra $\g h'$ of $\g g$.
\end{lemma}

\begin{proof}[Checking Lemma \ref{lemmamama}]
If this were not the case, one could find a sequence 
of semisimple Lie algebras 
$\g h\subsetneq \g h''\subsetneq\g h'\subsetneq \g g$,
each one being maximal in the next one, such that 
$\rho_{\g h}\not\leq\rho_{\g q}$. 
According to the previous discussion, the triple 
$(\g g,\g h',\g h'')$ has to be among the $5$ cases 
in Table \ref{figghh} with \lq\lq{N}\rq\rq\ in the last column,
{\it{i.e.}} $(\g f_4,\g b_4,\g d_4)$, 
$(\g e_6,\g d_5,\g b_4)$,
$(\g e_6,\g f_4,\g b_4)$,
$(\g e_7,\g a_1\g d_6,\g d_6)$, or
$(\g e_8,\g a_1\g e_7,\g e_7)$.
Since, one also has $\rho_{\g h}\not\leq \rho_{\g h'/\g h}$, 
there are very few possibilities for such an $\g h$. 
Here is the list of quadruples
$(\g g,\g h',\g h'',\g h)$~: \\
{\bf{Case 1.}} $(\g f_4,\g b_4,\g d_4,\g b_3)$, \\
{\bf{Case 2.}} $(\g e_6,\g d_5,\g b_4,\g d_4)$,\\
{\bf{Case 3.}} $(\g e_6,\g f_4,\g b_4,\g d_4)$,\\
{\bf{Case 4.}} $(\g e_7,\g a_1\!\oplus\!\g d_6,\g d_6,\g h)$, or\\
{\bf{Case 5.}} $(\g e_8,\g a_1\!\oplus\!\g e_7,\g e_7,\g h)$.
\vspace{1em}

In Case $1$, for all the possible embeddings $\g b_3\hookrightarrow\g d_4$
the representation $\g q$ of $\g h$ are isomorphic, hence we can assume that 
this embedding is the standard embedding. 
But since $\g h=\g b_3$ and $\g q= \m C\oplus V_1\oplus V_1\oplus V_2\oplus V_2$ with $V_1=\m C^7$ and $V_2 = \m C^8$ the spin representation, a direct computation gives  $p_{V_1 \oplus V_2}\le 2$,  and hence $p_{\g q}\leq 1$. Contradiction.

In Cases $2$ and $3$, 
 $\g q$ contains the sum of six $8$-dimensional irreducible 
representations $\m C^8$ of $\g d_4$. Since these representations $V$
satisfy $p_V=6$, one has $p_\g q\leq 1$. Contradiction.

Cases $4$ and $5$ are excluded because $\g h$ is included in $\g h'''=\g a_1\!\oplus\!\g h$ which is already excluded in Table  \ref{figghh}.
\end{proof}

\section{Bounding $p_V$ for simple Lie algebras}
\label{secbpvsim}

The aim of this chapter is to check Theorem \ref{thirrsim} 
that we used in the proof of Theorem \ref{thefg}.
This theorem \ref{thirrsim} follows from the concatenation of Propo\-sitions 
\ref{probpval} to \ref{probpvgl}. 

We will use freely the notations of Bourbaki \cite{Bou456, Bou78},
when describing the root system, 
 simple roots $\alpha_j$, 
 fundamental weights $\omega_j$, 
 and irreducible representations of a complex simple Lie algebra $\g h$.

When $\g h$ is a complex semisimple Lie algebra and $V$ a 
representation of $\g h$, 
 the function $\rho_V$, as in Section \ref{secrhoV},
 takes the form $\rho_V=\tfrac12\sum  m_\al |\al|$ 
 on a maximally split abelian 
real subalgebra $\g a$ of $\g h$.  
From now on, we will choose $m_\al$ to be the complex dimension of $V$
instead of the real dimension.
This modification of both $\rho_V$ and $\rho_\g h$ 
by a factor $\frac12$ is harmless since 
it does not affect the inequality $\rho_\g h\leq \rho_V$
or the value of $p_V$.
\vspace{1em}

The checking of the following seventeen propositions 
from \ref{probpval} to \ref{probpvgl} and 
from \ref{probpvaa1} to \ref{proslslslv} relies on explicit
 and about thirty-pages-long calculations
that we do not reproduce here.

\subsection{Bounding $p_V$ for $\g a_\ell$}
\label{secbpval}

In this section $\g h$ is the complex simple Lie algebra 
$\g h=\g a_\ell=\g s\g l_{\ell+1}$
 with $\ell \ge 2$.  
The case $\ell=1$
will be treated in Corollary \ref{corbpvaabg} when $V$ is not necessarily irreducible.
\begin{pro}
\label{probpval}
Let $\g h=\g a_\ell$
 with $\ell \ge 2$,
 and $V$ be an irreducible faithful represen\-tation of $\g h$ 
such that $p_V>1$, 
 equivalently, 
 $\rho_{\g h}\not\leq\rho_V$, 
then $V$ or $V^*$ is either
\\  
$V_{\om_1}=\m C^{\ell+1}$ 
and $p_V=2\ell$, 
or\\ 
$V_{2\,\om_1}=S^2\m C^{\ell+1}$ and $p_V=2\, \tfrac{\ell}{\ell+1}<2$, 
 or
\\
$V_{\om_2}=\La^2\m C^{\ell+1}$ and $p_V= 2 \tfrac{\ell+2}{\ell}$ for $\ell$ even and $p_V= 2 \tfrac{\ell+1}{\ell-1}$ for $\ell$ odd, 
or
\\
$V_{\om_3}=\La^3\m C^{\ell+1}$
 when $\ell =3$, $4$, $5$, $6$, $7$, 
 and $p_V=6$, $3$, $2$, $\tfrac{10}{7}$, $\tfrac{10}{9}$, 
 respectively.  
\end{pro}


\subsection{Bounding $p_V$ for $\g b_\ell$}
\label{secbpvbl}
In this section $\g h$ is the complex simple Lie algebra 
$\g h=\g b_\ell=\g s\g o_{2\,\ell+1}$.

\begin{pro}
\label{probpvbl}
Let $\g h=\g b_\ell$
with $\ell\geq 2$,
 and $V$ be an irreducible faithful represen\-tation of $\g h$ 
such that $\rho_{\g h}\not\leq\rho_V$, 
then $V$  is either
\\  
$V_{\om_1}=\m C^{2\,\ell+1}$ and $p_V=2\,\ell\!-\! 1$,
 or
\\ 
$V_{\om_\ell}=\m C^{2^\ell}$  when $\ell =2$, $3$, $4$, $5$, $6$
and $p_V=4$, $4$, $3$, $2$, $\tfrac{4}{3}$ respectively.  
\end{pro}


\subsection{Bounding $p_V$ for $\g c_\ell$}
\label{secbpvcl}
In this section $\g h$ is the complex simple Lie algebra 
$\g h=\g c_\ell=\g s\g p_{\ell}$.

\begin{pro}
\label{probpvcl}
Let $\g h=\g c_\ell$
with $\ell\geq 3$,
 and $V$ be an irreducible faithful represen\-tation of $\g h$ 
such that $\rho_{\g h}\not\leq\rho_V$, 
then $V$  is either\\  
$V_{\om_1}=\m C^{2\,\ell}$ and $p_V=2\,\ell$,
 or
\\ 
$V_{\om_2}=\La^2_0\m C^{2\,\ell}$ and $p_V=\tfrac{\ell+1}{\ell-1}$, 
or
\\ 
$V_{\om_\ell}=\La^3_0\m C^{2\ell}$  when $\ell =3$
and $p_V=\tfrac{5}{3}$.  
\end{pro}


\subsection{Bounding $p_V$ for $\g d_\ell$}
\label{secbpvdl}

In this section $\g h$ is the complex simple Lie algebra 
$\g h=\g d_\ell=\g s\g o_{2\,\ell}$.

\begin{pro}
\label{probpvdl}
Let $\g h=\g d_\ell$
with $\ell\geq 4$,
 and $V$ be an irreducible faithful represen\-tation of $\g h$ 
such that $\rho_{\g h}\not\leq\rho_V$, 
then $V$  is either
\\  
$V_{\om_1}=\m C^{2\,\ell}$ and $p_V=2\,\ell\!-\! 2$, 
 or
\\ 
$V_{\om_{\ell-1}}$ or $V_{\om_\ell}=\m C^{2^{\ell-1}}$  when $\ell =4$,  $5$, $6$, $7$
and $p_V=6$, $\tfrac{7}{2}$, $\tfrac{5}{2}$, $\tfrac{3}{2}$ respectively.
\end{pro}


\subsection{Bounding $p_V$ for $\g e_\ell$}
\label{secbpvel}

In this section $\g h$ is the complex simple Lie algebra 
$\g h=\g e_\ell$.

\begin{pro}
\label{probpvel}
Let $\g h=\g e_\ell$
with $\ell= 6$, $7$ or $8$
 and $V$ be an irreducible faithful representation of $\g h$ 
such that $\rho_{\g h}\not\leq\rho_V$, 
then $V$  is either\\  
$V_{\om_1}$ or $V_{\om_6}=\m C^{27}$ when $\ell=6$ and $p_V=\tfrac{7}{2}$, or\\ 
$V_{\om_{7}}=\m C^{56}$  when $\ell =7$
and $p_V=\tfrac{17}{6}$.  
\end{pro}


\subsection{Bounding $p_V$ for $\g f_4$}
\label{secbpvfl}

In this section $\g h$ is the complex simple Lie algebra 
$\g h=\g f_4$.

\begin{pro}
\label{probpvfl}
Let $\g h=\g f_4$ and $V$ be an irreducible faithful representation of $\g h$ 
such that $\rho_{\g h}\not\leq\rho_V$, 
then
 $V=V_{\om_4}=\m C^{26}$ 
and $p_V=\tfrac{8}{3}$.
\end{pro}


\subsection{Bounding $p_V$ for $\g g_2$}
\label{secbpvgl}

In this section $\g h$ is the complex simple Lie algebra 
$\g h=\g g_2$.

\begin{pro}
\label{probpvgl}
Let $\g h=\g g_2$ and $V$ be an 
irreducible faithful representation of $\g h$ 
such that $\rho_{\g h}\not\leq\rho_V$, 
then
 $V=V_{\om_1}=\m C^{7}$ 
and $p_V=3$.
\end{pro}


\subsection{Bounding $p_V$ for $\g a_1$, $\g a_2$, $\g b_2$, $\g g_2$}
\label{secbpvaabg}

{}From the discussion in this chapter,
 we get from \eqref{eqnpvpvpv} 
the following bound for 
$p_V$ when $V$ is not assumed to be irreducible.

\begin{cor}
\label{corbpvaabg}
Let $\g h$ be a simple Lie algebra and $V$ a representation of $\g h$ 
without nonzero $\g h$-invariant vector. Assume that either
$\g h=\g a_1$ and $\dim V\geq 3$, or
$\g h=\g a_2$ and $\dim V\geq 11$, or
$\g h=\g b_2$ and $\dim V\geq 15$, or
$\g h=\g g_2$ and $\dim V\geq 21$,
then one has  $\rho_{\g h}\leq\rho_V$.
\end{cor}


\section{Bounding $p_V$ for non-simple Lie algebras}
\label{secbpvnon}

In the previous chapters we used quite a few upper bounds 
for the invariant $p_V$ of various representations $V$
of semisimple Lie algebras $\g h$.
The aim of this chapter is to state precisely these upper bounds.

\subsection{Bounding $p_V$ for $\g a_1\oplus\g h_2$}
\label{secbpvaaabg}

In this section $\g h$ is a semisimple Lie algebra of the form 
$\g h=\g h_1\!\oplus\!\g h_2$ with $\g h_1=\g a_1$
and ${\rm rank}\, \g h_2\leq 2$. 
We want to bound $p_V$ when 
$V$ is a representation of $\g h$ such that, for $i=1$, $2$,  
the spaces $V^{\g h_i}$ of 
$\g h_i$-invariant vectors are $0$.

\begin{pro}
\label{probpvaa1}
Let $\g h=\g a_1\!\oplus\!\g a_1$  and 
$V$ be an irreducible faithful represen\-tation of $\g h$ 
such that $\rho_{\g h}\not\leq\rho_V$, 
then $V=\m C^2\otimes \m C^2$ 
and $p_V=2$.
\end{pro}


\begin{pro}
\label{probpvaa2}
Let $\g h=\g a_1\!\oplus\!\g a_2$  and 
$V$ be an irreducible faithful represen\-tation of $\g h$ 
such that $\rho_{\g h}\not\leq\rho_V$, 
then either\\
$V=\m C^2\otimes \m C^3$ or $\m C^2\otimes(\m C^3)^*$  
and $p_V=2$, or\\ 
$V=S^2\m C^2\otimes \m C^3$ or $S^2\m C^2\otimes(\m C^3)^*$
and $p_V=\tfrac{4}{3}$. 
\end{pro}


\begin{pro}
\label{probpvab2}
Let $\g h=\g a_1\!\oplus\!\g b_2$  and 
$V$ be an irreducible faithful represen\-tation of $\g h$ 
such that $\rho_{\g h}\not\leq\rho_V$, 
then either\\
$V=\m C^2\otimes \m C^4$  
and $p_V=2$, or\\ 
$V=S^2\m C^2\otimes \m C^4$
and $p_V=\tfrac{4}{3}$, or\\
$V=\m C^2\otimes \m C^5$ 
and $p_V=\tfrac{3}{2}$. 
\end{pro}


\begin{pro}
\label{probpvag2}
Let $\g h=\g a_1\!\oplus\!\g g_2$  and 
$V$ be an irreducible faithful represen\-tation of $\g h$ 
such that $\rho_{\g h}\not\leq\rho_V$, 
then 
$V=\m C^2\otimes \m C^7$  
and $p_V=\tfrac{3}{2}$. 
\end{pro}


From the discussion in this section, we get the following bound for 
$p_V$ when $V$ is not assumed to be irreducible.

\begin{cor}
\label{corbpvaaabg}
Let $\g h_1=\g a_1$, 
$\g h_2$ be a simple Lie algebra,
$\g h=\g h_1\!\oplus\!\g h_2$ and 
$V$ a representation of $\g h$ 
without $\g h_1$-invariant vector or $\g h_2$-invariant vector. 
Assume that either
$\g h=\g a_1\!\oplus\!\g a_1$ and $\dim V\geq 6$, or
$\g h=\g a_1\!\oplus\!\g a_2$ and $\dim V\geq 12$, or
$\g h=\g a_1\!\oplus\!\g b_2$ and $\dim V\geq 15$, or
$\g h=\g a_1\!\oplus\!\g g_2$ and $\dim V\geq 21$,
then one has  $\rho_{\g h}\leq\rho_V$.
\end{cor}


\subsection{Bounding $p_V$ for $\g h_1\oplus\g h_2$}
\label{secbpvhh}

The following proposition is a reformulation of Proposition \ref{proclasym1}.

\begin{proposition}
\label{proclasym1v}
Let $p\geq 1$ and $q\geq 1$. \\
$\bullet$ Let $\g h=\g s\g l_p\oplus\g s\g l_q$ acts on $V=\m C^p\otimes\m C^q$.
Then 
$\rho_\g h\not\leq 2\rho_V\Leftrightarrow |p-q|\geq 2$.\\
$\bullet$ Let $\g h=\g s\g o_p\oplus\g s\g o_q$ acts on $V=\m C^p\otimes\m C^q$.
Then 
$\rho_\g h\not\leq \rho_V\Leftrightarrow |p-q|\geq 3$.\\
$\bullet$ Let $\g h=\g s\g p_p\oplus\g s\g p_q$ acts on $V=\m C^{2p}\otimes\m C^{2q}$.
Then one has
$\rho_\g h\not\leq \rho_V$.
\end{proposition}


\subsection{Bounding $p_V$ for tensor products}
\label{secbpvtp}

The following proposition is a reformulation  of Proposition \ref{proclaten}.

\begin{proposition}
\label{proclatenv}
Suppose $p> 1$ and $q>1$.\\
$\bullet$ If $\g h=\g s\g l_p\oplus\g s\g l_q$ acts on 
$V={\rm End}_0\m C^p\otimes{\rm End}_0\m C^q$, then $\rho_\g h\leq \rho_V$.\\
$\bullet$ If $\g h=\g s\g o_p\oplus\g s\g o_q$ acts on 
$V=\La^2\m C^p\otimes S^2_0\m C^q\oplus S^2_0\m C^p\otimes \La^2\m C^q$, 
then $\rho_\g h\leq \rho_V$.\\
Suppose $p\geq 1$ and $q>1$.\\
$\bullet$ If $\g h=\g s\g p_p\oplus\g s\g p_q$ acts on 
$V=S^2\m C^{2p}\otimes \La^2_0\m C^{2q}\oplus 
\La^2_0\m C^{2p}\otimes S^2\m C^{2q}$,
then $\rho_\g h\leq \rho_V$.\\
$\bullet$ If $\g h=\g s\g p_p\oplus\g s\g o_q$ acts on 
$V=S^2\m C^{2p}\otimes S^2_0\m C^{q}\oplus 
\La^2_0\m C^{2p}\otimes \La^2\m C^{q}$,
then $\rho_\g h\leq \rho_V$.
\end{proposition}


The following proposition is a reformulation of Proposition \ref{proslspsl}.

\begin{proposition}
\label{proslspslv}
Let $p\geq 1$, $q\geq 1$. \\
$\bullet$ Let $\g h=\g s\g p_p\!\oplus\!\g s\g l_q$ acts on 
$V=\La^2_0\m C^{2p}\oplus
(\m C^{2p}\!\otimes\!\m C^q\oplus dual)$.\\
Then 
$\rho_\g h\not\leq \rho_V\Leftrightarrow q\geq 2p\!+\!2$.\\
$\bullet$ Let $\g h=\g s\g l_p\oplus\g s\g o_q$ acts on 
$V=\La^2\m C^p\oplus\m C^{p}\!\otimes\!\m C^q$.
Then 
$\rho_\g h\not\leq 2\rho_V\Leftrightarrow q\geq 2p+3$.\\
$\bullet$ Let $\g h=\g s\g l_p\oplus\g s\g p_q$ acts on 
$V=S^2\m C^p\oplus\m C^{p}\!\otimes\!\m C^{2q}$.
Then 
$\rho_\g h\not\leq 2\rho_V\Leftrightarrow q\geq p+1$.\\
$\bullet$ Let $\g h=\g s\g p_p$ acts on 
$V=\La^2_0\m C^{2p}$ 
and $p\geq 2$.
Then 
$\rho_\g h\leq 3\rho_V$.
\end{proposition}


The following proposition is a reformulation of Proposition \ref{prososo}.

\begin{proposition}
\label{prososov}
Let  $q\geq 1$.\\
$\bullet$ Let $\g h=\g g_2\oplus\g s\g o_{q}$ act on 
$V=\m C^7\otimes (\m C\oplus \m C^q)$ via $\g g_2\stackrel{irr}{\hookrightarrow} \g s\g o_7$.\\
Then 
$\rho_\g h\not\leq \rho_V\Leftrightarrow q=1$ or $q\geq 10$.\\
$\bullet$ Let $\g h=\g s\g o_{7}\oplus\g s\g o_{q}$ act on 
$V=\m C^8\otimes (\m C\oplus \m C^q)$ via
$\g s\g o_7\stackrel{irr}{\hookrightarrow} \g s\g o_8$.\\
Then 
$\rho_\g h\not\leq \rho_V\Leftrightarrow q=1$, $q=2$, or $q\geq 11$.
\end{proposition}


\subsection{Bounding $p_V$ for $\g h_1\oplus\g h_2\oplus\g h_3$}
\label{secbpvhhh}

The following proposition was used in the proof of Proposition \ref{proslslsl}.

\begin{proposition}
\label{proslslslv}
Let $p\geq q\geq r\geq 1$. \\
$\bullet$ Let $\g h=\g s\g l_p\oplus\g s\g l_q\oplus\g s\g l_r$ 
act on $V=\m C^p\!\otimes\!(\m C^q)^*\oplus 
\m C^q\!\otimes\!(\m C^r)^*\oplus \m C^p\!\otimes\!(\m C^r)^*$.
Then 
$\rho_\g h\leq 2\rho_V\Leftrightarrow p\leq q+r+1$.\\
$\bullet$ Let $\g h=\g s\g o_p\oplus\g s\g o_q\oplus\g s\g o_r$ 
act on $V=\m C^{p}\!\otimes\!\m C^{q}\oplus 
\m C^{q}\!\otimes\!\m C^{r}\oplus \m C^{p}\!\otimes\!\m C^{r}$.
Then 
$\rho_\g h\leq \rho_V\Leftrightarrow p\leq q+r+2$.\\
$\bullet$ Let $\g h=\g s\g p_p\oplus\g s\g p_q\oplus\g s\g p_r$ 
act on $V=\m C^{2p}\!\otimes\!\m C^{2q}\oplus 
\m C^{2q}\!\otimes\!\m C^{2r}\oplus \m C^{2p}\!\otimes\!\m C^{2r}$.
Then 
$\rho_\g h\leq \rho_V\Leftrightarrow p\leq q+r$.
\end{proposition}


\section{Real reductive Lie algebras}
\label{secrearedsub}

The aim of this chapter is to check Theorem \ref{thrhrqr}.
We
note  that Theorem \ref{thrhrqr} allows us  to give a complete description of the pairs $G \supset H$ of real reductive algebraic Lie groups
for which $L^2(G/H)$ is not tempered.  
In fact, 
 let $\g g$ be a real reductive Lie algebra, $\g h$ a reductive 
Lie subalgebra of $\g g$ and $\g q=\g g/\g h$. 
By the criterion \eqref{eqnl2ghrhrq}, 
 we want to classify the pairs $(\g g,\g h)$
 such that $\rho_\g h\not\leq\rho_\g q$. 
According to Lemma \ref{lemrhrqreacom} and Proposition \ref{prorhrqreacom}, we can assume that $\g h$ is semisimple without ${\rm Ad}$-compact ideals 
and that $\g g$ is simple.

To prove Theorem \ref{thrhrqr},
 we recall that either the simple Lie algebra $\g g$ has a complex structure or $\g g$ is 
absolutely simple {\it{i.e.}} the complexified 
Lie algebra $\g g_\m C$ is simple.
We deal the first case in Section \ref{secqq}, 
and the second case
in Sections \ref{secfinwit}--\ref{secproreacla}.  

\subsection{When $\g g$ is a complex Lie algebra}
\label{secqq}

We first deal with the case where $\g g$ has a complex structure.

\begin{proposition}
\label{progchr}
Let $\g g$ be a complex simple Lie algebra
 with complex structure $J$, 
 and $\g h$ a real semisimple Lie subalgebra of $\g g$ 
such that $\rho_{\g h}\not\leq \rho_{\g g/\g h}$.
Then the complex Lie subalgebra $\g h_0:=\g h\cap J{\g h}$ 
also satisfies $\rho_{\g h_{_0}}\not\leq \rho_{\g g/\g h_{_0}}$.
\end{proposition}

\begin{proof}
The complex subspace $\g h_0$ is indeed an ideal of $\g h$.  
Since $\g h$ is semisimple, 
 it decompose into the direct sum $\g h= \g h_0 \oplus \g h_1$
of two ideals $\g h_0$ and $\g h_1$, 
 where the semisimple ideal $\g h_1$ is totally real in $\g h$, 
 {\it{i.e.,}}
 $\g h_1 \cap J \g h_1= \{0\}$.  
We set $\widetilde{\g h}:= \g h \oplus J \g h_1=\g h_0 \oplus \g h_1 \oplus J \g h_1$.  
Then $\widetilde {\g h}$ is a complex subalgebra of $\g g$.

Assume there exists $X =X_0 + X_1 \in \g h$
 such that $\rho_{\g h}(X)> \rho_{\g g/\g h}(X)$.  
We claim  $\rho_{\g h_0}(X_0)> \rho_{\g g/\g h_0}(X_0)$.  
Indeed, 
 since $[\g h_0, \g h_1]=\{0\}$, 
 one has
\[
   \rho_{\g h_0}(X_0) = \rho_{\g h}(X)-\rho_{\g h_1}(X_1) 
   >
  (\rho_{\g g/\widetilde {\g h}}(X) + \rho_{J\gs h_1}(X_1))-\rho_{\g h_1}(X_1)
= \rho_{\g g/\widetilde {\g h}}(X).  
\]
Using Lemma \ref{lemv1v2} (1), one goes on
 $\rho_{\g g/\widetilde {\g h}}(X) \ge \rho_{\g g/\widetilde {\g h}}(X_0) = \rho_{\g g/\g h_0}(X_0)$.  
Therefore, one gets $\rho_{\g h_0} \not \le \rho_{\g g/\g h_0}$.  
\end{proof}

By Proposition \ref{progchr}, 
 Theorem \ref{thrhrqr} in the case
 where $\g g$ has a complex structure is deduced from 
Theorem \ref{thrhrqc}.

Moreover,
 Proposition \ref{progchr} implies that the list of such pairs $(\g g,\g h)$ 
are given by Tables \ref{figabcd} and \ref{figefg} with the following
two modifications:
In Table \ref{figabcd}, one allows $\g h_2$ to be  real Lie subalgebras,
and, in Table \ref{figefg},
 one allows pairs $(\g e_7,\g d_6\!\oplus\!\g h_2)$ and $(\g e_8,\g e_7\!\oplus\!\g h_2)$
with $\g h_2\subset\g s\g l_2$.


\subsection{Finding a witness in $\g a_\g h$}
\label{secfinwit}

We assume in this section that $\g g$ is an absolutely simple Lie algebra,
 and that $\g h$ is a semisimple Lie subalgebra of $\g g$.  
We denote by $\g g_\m C$, $\g h_\m C$ and $\g q_\m C$ the complexifications
of $\g g$, $\g h$ and $\g q$. 
According to  Lemma \ref{lemrhrvc}, one  has 
$\rho_{\g h_{_\m C}}\not\leq \rho_{\g q_{_\m C}}$ if $\rho_{\g h} \not \le \rho_{\g q}$.  
According to Theorems \ref{thabcd} and \ref{thefg}, 
the pair $(\g g_\m C,\g h_\m C)$ satisfying $\rho_{\g h_{_\m C}}\not\leq \rho_{\g q_{_\m C}}$
 has to be in Tables
\ref{figabcd} or \ref{figefg}.
\vspace{1em}

We consider the $\g h$-module $V:=[\g h,\g q]$ and its complexification 
$V_\m C$.
We note 
 that $\rho_{\g q} = \rho_{V}$ and $\rho_{\g q_{\m C}} = \rho_{V_{\m C}}$, 
 hence $\rho_{\g h} \le \rho_{\g q}$ $\Leftrightarrow$ $\rho_{\g h} \le \rho_V$ $\Leftrightarrow$ $2\rho_{\g h} \le \rho_{V_{\m C}}$ on $\g h$
 and $\rho_{\g h_{\m C}} \le \rho_{\g q_{\m C}}$ $\Leftrightarrow$ $\rho_{\g h_{\m C}} \le \rho_{V_{\m C}}$ on $\g h_{\m C}$.  
For a while, we will forget $\g g$ and  just remember 
the list of representations $(\g h_\m C,V_\m C)$.  
For each case in Tables \ref{figabcd} and \ref{figefg},
 we look for the minimal $\g h_\m C$
 and we report the corresponding 
representation in Tables \ref{figabcdr} and \ref{figefgr}.

For the representations $(\g h_{\m C}, V_{\m C})$
 with $\rho_{\g h_{\m C}} \not \le \rho_{V_{\m C}}$
 in Tables \ref{figabcdr} and \ref{figefgr}, 
 we want first to know whether one can find a real form 
$\g h$ of $\g h_\m C$
 and an $\g h$-invariant real form $V$ of $V_{\m C}$
 such that $\rho_\g h\leq \rho_V$.
The answer is most often {\it{No}} 
 but there are a few exceptions.  
To see this, 
 we introduce useful notion 
 that helps us to find when $\rho_{\g h_{\m C}} \not \le \rho_{V_{\m C}}$
 implies $\rho_{\g h} \not \le \rho_V$:
\begin{definition}
[witness]
\label{def:witness}
Let $V$ be an $\g h$-module.  
We say a vector $X$ in $\g h$
 is a {\it{witness}}
 if $\rho_{\g h}(X) > \rho_V(X)$.  
We denote by $\operatorname{Wit}(\g h, V)$
 the subset of $\g h$ consisting of witness vectors.  
\end{definition}

By definition,
 $\operatorname{Wit}(\g h, V) \ne \emptyset$ 
 if and only if $\rho_{\g h} \not \le \rho_V$.  
If $(\g h_{\m C}, V_{\m C})$ is the complexification of $(\g h, V)$, 
then one has
\begin{equation}
\label{eqn:witnessreal}
 \operatorname{Wit}(\g h_{\m C}, V_{\m C}) \cap \g h
 =
 \operatorname{Wit}(\g h, V).  
\end{equation}

Back to our setting
 where $\rho_{\g h_{\m C}} \not \le \rho_{V_{\m C}}$, 
 we choose a Cartan subalgebra $\g j_{\m C}$
 of the complex semisimple Lie algebra $\g h_{\m C}$
 such that $\g j_{\m C} \cap \g h$ contains a maximal split abelian subalgebra $\g a_{\g h}$
 of a real form $\g h$ of $\g h_{\m C}$.  
We know that there exists a witness $X$ in $\g j_{\m C}$
{\it{i.e.}} an element such that  
$\rho_{\g h_{_\m C}}(X)> \rho_{V_{_\m C}}(X)$.  
We shall see that we can find a witness $X$ in $\g a_{\g h}$
 for most of noncompact real forms $\g h$ of $\g h_{\m C}$.  
More precisely,
 one has the following lemma:

\begin{lemma}
\label{lemabcdefgr}
Let $\g h$ be a real semisimple Lie algebra without ${\rm Ad}$-compact ideals, 
and $V$ a representation of $\g h$ over $\m R$. 
Assume that the pair $(\g h_\m C, V_\m C)$
is in Tables \ref{figabcdr} or \ref{figefgr}.
Then one has $\rho_\g h\not\leq \rho_V$, 
except\\ 
in Case $A1$ with $p-q=2$, $p=2p'$ and $\g h=\g s\g l(p', \m H)$ or\\
in Cases $D5$, $E6.1$, $E6.2$, $E7.1$ with ${\rm rank}_\m R\,\g h=1$.
\end{lemma}

\begin{figure}[htb]
\begin{center}
\begin{tabular}{|c|c|c|c|c|c|}
\hline
\small\!\!Case\!\!&$\g h_\m C$&$V_\m C$
&\small\!parameters\!&\small witnesses &\!\!$\rho_\g h\!\leq \!\rho_{V}$\!\!\!
\\
\hline
$A1$&$\g s\g l_{p}$&$(\m C^p\!\oplus\! \text{dual})^q
$\!&\!$
p=q+2$\!&$\scriptstyle(1,0,\ldots, 0,-1)$
&\!\!$\scriptstyle\g h=\g s\g l(m,\m H)\; p=2m$\!\!\\
&&&\!$p\geq q+3
$\!
&\begin{tabular}{l}\raisebox{-0.2em}{$\scriptstyle(1,0,0,\ldots, 0,0,-1)\;{\rm and}\,$}\\ $\scriptstyle(1,1,0,\ldots,0,-1,-1)$\end{tabular}\!\!&No
\\
\hline
\!$B\!D1$\!&$\g s\g o_{p}$&$(\m C^p)^q
$&\!$
p= q+3$\!&\!\!$\scriptstyle(1,0,0,\ldots,0)$\!
&No
\\
&
&
&\!
$p\geq q+4$\!
&\begin{tabular}{l}\raisebox{-0.2em}{$\scriptstyle(1,0,0,\ldots,0)\;{\rm and}\,$}
\\
$\scriptstyle(1,1,0,\ldots,0)$\end{tabular}\!\!
&No
\\
\hline
$C1$&$\g s\g p_{p}$&$(\m C^{2p})^q
$&\!$p\geq q+1
$\!&$
\scriptstyle(1,0,0,\ldots)\;{\rm and}\; (1,1,0,0,\ldots)$&No\\
\hline
$A2$&$\g s\g p_{p}$&$\La^2_0\m C^{2p}$&
$p\geq 2$&$\scriptstyle{\rm any}\; X\neq 0$&No\\
\hline
$C2$&\!\!$\g s\g p_{p}\!\!\oplus\!\!\g s\g p_p$\!\!&$\m C^{2p}\otimes\m C^{2p}$
&
$p\geq 2$&$\scriptstyle(Y,Y)\;{\rm  for\;\; any}\;\; Y\neq 0$&No\\
\hline
$D2$&$\g s\g l_{p}$&$\La^2\m C^p\!\oplus\! \text{dual}$&
$p=3$
&$\scriptstyle{\rm any}\; X \neq 0$&No\\
&&&
$p\geq 4$
&\begin{tabular}{l}\raisebox{-0.2em}{$\scriptstyle(1,0,0,\ldots, 0,0,-1)\;{\rm and}\,$}\\ $\scriptstyle(1,1,0,\ldots,0,-1,-1)$\end{tabular}&No\\
\hline
$B3$&$\g g_2$&$\m C^7$&&
$\scriptstyle{\rm any}\; X\neq 0$&No\\
\hline
$B4$&$\g g_2$&$\m C^7\oplus\m C^7$&&
$\scriptstyle{\rm any}\; X\neq 0$&No\\
\hline
$D4$&$\g s\g o_7$&$\m C^7\oplus\m C^8$&&
$\scriptstyle{\rm any}\; X\neq 0$&No\\
\hline
$D5$&$\g s\g o_7$&\!$\m C^7\!\oplus\!\m C^8\!\oplus\! \m C^8$\!&&
$\scriptstyle (1,1,0)$&\scriptsize{$\operatorname{rank}_\m R\,\g h\;=\; 1$}\\
\hline
\end{tabular}
\end{center}
\caption{ Representations $V_\m C$
 of $\g h_\m C$ when $\g g_\m C$ is classical}
\label{figabcdr}
\end{figure}

\begin{figure}[htb]
\begin{center}
\begin{tabular}{|c|c|c|c|c|}
\hline
\small\!\!Case\!\!&$\g h_\m C$&$V_\m C$&\small witnesses 
&\!\!$\rho_\g h\!\leq \!\rho_{V}$\!\!\!\\
\hline
$G2$&$\g s\g l_{3}$&$\m C^3\!\oplus\! \text{dual}$
&$\scriptstyle{\rm any}\; X\neq 0$&No
\\
\hline
$F4$&$\g s\g o_8$&\!$\m C^8\!\oplus\!\m C^8\!\oplus\! \m C^8$\!&
$\scriptstyle{\rm any}\; X\neq 0$
&No
\\
\hline
$E6.1.a$&$\g s\g o_{10}$&$\m C^{16}\oplus \text{dual}$&
$\scriptstyle (1,1,0,0,0)$&$\scriptstyle \operatorname{rank}_\m R\,\g h\;=\; 1$\\
\hline
$E6.2.a$&$\g f_{4}$&$\m C^{26}$&
$\scriptstyle{\rm any}\; X\neq 0$
&No\\
\hline
\!\!$\begin{tabular}{c}$E6.1.b$\\$E6.2.b$\end{tabular}$\!\!&$\g s\g o_{9}$
&$\m C^{9}\oplus \m C^{16}\oplus\m C^{16}$&
$\scriptstyle (1,1,0,0)$&$\scriptstyle \operatorname{rank}_\m R\,\g h\;=\; 1$\\
\hline
$E7.1$&$\g s\g o_{12}$&$\m C^{32}\oplus \m C^{32}$&
$\scriptstyle (1,1,0,0,0,0)$&$\scriptstyle \operatorname{rank}_\m R\,\g h\;=\; 1$\\
\hline
$E7.2$&$\g e_{6}$&$\m C^{27}\oplus \text{dual}$&
$\scriptstyle
(0,0,0,0,1,-1,-1,1)$&No\\
\hline
$E8$&$\g e_{7}$&$\m C^{56}\oplus \m C^{56}$&
$\scriptstyle 
(0,0,0,0,1,1,-1,1)$&No\\
\hline
\end{tabular}
\end{center}
\caption{ Representations $V_\m C$ of $\g h_\m C$ when $\g g_\m C$ is exceptional}
\label{figefgr}
\end{figure}

Here are a few comments on Tables \ref{figabcdr} and \ref{figefgr}~:
\\
- The name of the cases in the first column are those from Tables
\ref{figabcd} and \ref{figefg}.\\ 
- In the third column each $\m C^n$ stands for an irreducible representation 
of $\g h_\m C$.\\
-  In Case $F4$ the representations $\m C^8$ are the three 
distinct $8$-dimensional representations of $\g s\g o_8$.\\
- In the last column we describe all the 
real form $\g h$ of $\g h_\m C$ without ${\rm Ad}$-compact ideal  
for which $\rho_{\g h} \le \rho_V$.  
\\
- The answer {\it{No}} indicates that 
 such $\g h$ does not exist.  
\\ 
- The notation for the witness in the Cartan subspace of $\g h_\m C$
 uses the standard basis
 with notation as in \cite[Chap.~6]{Bou456}.  
\\
- For most of the case, we  only reported
 in Tables \ref{figabcdr} and \ref{figefgr} the Lie subalgebra $\g h_\m C$
 which are minimal in the case,
 since when $\rho_{\g h} \le \rho_V$
 fails for $\g h$,
 so does it for any larger subalgebras.
\\ 
- We will see that Cases $A1$, $E6.1.a$, $E6.1.b$ and $E6.2.b$ correspond to Cases 
$(ii)$, $(iii)$ and $(iv)$, 
 respectively
 in Theorem \ref{thrhrqr}.\\
- We will see that Cases $D5$ and $E7.1$ cannot occur 
from a pair $(\g g,\g h)$ when ${\rm rank}_\m R\g h =1$.

\begin{proof}[Checking Lemma \ref{lemabcdefgr}]
By a direct case-by-case calculation one sees that the vectors in the fourth column are witnesses.

According to the classification of real forms of semisimple Lie algebras,
for a given semisimple Lie algebra $\g h_\m C$ the various 
Cartan subspaces $\g a_\g h\subset \g j_{\m C}$ 
of real forms $\g h$ are described by the Satake diagrams 
(see \cite[pp.~532--534]{He78}).
For any real form $\g h$, one can often choose one of the witness
in the fourth column to be in $\g a_{\g h}$.
The only exceptions are the ones indicated in the last column
 or $BD1$ with $p-q=3$, 
 $p=2p'$
 and $\g h= \mathfrak{s o}^{\ast}(2p')$. 
 
In this latter case $BD1$
where $\g h=\mathfrak{s o}^{\ast}(p)$ ($p$:even), one has 
$\operatorname{Wit}(\g h, V_{\m C})\cap \g h=\emptyset$.
However, 
we can exclude this case
because the $\g h_{\m C}$-module $V_{\m C}=(\m C^p)^{p-3}$ is not defined
over $\m R$.

Finally we check that indeed the remaining 
Cases $A1$, $D5$, $E6.1$, $E6.2$ and $E7.1$
satisfy $\rho_\g h\leq \rho_{V}$.  
\end{proof}

Now we want to detect whether 
these remaining Cases $A1$, $D5$, $E6.1$, $E6.2$ and $E7.1$
can arise from a pair $(\g g,\g h)$ with $\rho_\g h\leq \rho_\g q$.

\subsection{Subalgebras defined over $\m R$}
\label{secreamax}

For all pairs $(\g g_\m C,\g h_\m C)$ which occur
in  Tables \ref{figabcd} and \ref{figefg}, the Lie subalgebra 
$\g h_\m C$ is included in a maximal semisimple subalgebra 
of $\g g_\m C$. Moreover, except for Case $G_2$, 
this maximal Lie subalgebra is the derived algebra of 
a symmetric Lie subalgebra of $\g g_\m C$. 
We first need to know that all these Lie subalgebras 
are defined over $\m R$. 
This will follow from the general lemmas below: 

\begin{lemma}
\label{lemreamax}
Let $\g g$ be a real simple Lie algebra and $\g h$ a maximal 
real semi\-simple Lie subalgebra of $\g g$. 
Then $\g h_\m C$ is a maximal complex semisimple Lie subalgebra of $\g g_\m C$.
\end{lemma}


\begin{lemma}
\label{lemreasym}
Let $\g g$ be a real simple Lie algebra and $\g l$ 
a symmetric Lie subalgebra of $\g g_\m C$. 
If the semisimple Lie algebra $[\g l,\g l]$ is defined over $\m R$
and $\g g\neq \g s\g l(2,\m R)$, $\g g\neq \g s\g l(2,\m C)$
then $\g l$ is also defined over $\m R$. 
\end{lemma}


\subsection{Checking Theorem \ref{thrhrqr}}
\label{secproreacla}

Two points remain to be checked
 when $\g g$ is absolutely simple.

It remains to check that Cases $A1$, $E6.1.a$, $E6.1.b$ and $E6.2.b$ correspond to Cases 
$(ii)$, $(iii)$ and $(iv)$, 
 respectively, 
 in Theorem \ref{thrhrqr}.
This follows from Lemmas \ref{lemreamax} and \ref{lemreasym}
and from Berger's classification of irreducible real symmetric spaces
 \cite{xber}.

It remains also to check that Cases $D5$ and $E7.1$ cannot occur 
from a pair $(\g g,\g h)$ when ${\rm rank}_\m R\g h =1$.

In Case $D5$, one has $\g g_\m C=\g s\g o_{10}$ and $\g h=\g s\g o (6,1)$.
By repeated applications of Lemma \ref{lemreamax},
the representation of $\g h$ in $\g q$ must be a direct sum
of irreducible representations
$\g q=\m R^7\oplus\m R^8\oplus(\m R\oplus \m R^8)$.
This contradicts the fact that, for $\g h=\g s\g o(p,q)$ with  $n=p+q$ odd,
the spin representation of $\g h_\m C$ 
can be defined over $\m R$ only if $p-q=\pm 1\;{\rm mod}\;8$.

In Case $E7.1$, one has $\g g_\m C=\g e_7$ and $\g h=\g s\g o(11,1)$. 
According to Lemma \ref{lemreamax}, the Lie subalgebra
$\g h$ is included in a subalgebra $\g h'$ of $\g g$ such that
$\g h'_\m C=\g d_6\oplus\g a_1$
But according to Berger's classification of real symmetric spaces,
the complex symmetric pair $(\g e_7,\g d_6\oplus\g a_1)$
has only four real forms $(\g g,\g h')$ and none of the $\g h'$ 
contains $\g s\g o(11,1)$.

\section{Reductive homogeneous spaces}
\label{secredhom}

In this chapter, we come back to the 
point of view of Lie groups and their homogeneous spaces $G/H$.

We first relate in Section \ref{secgenstaqgh} the generic stabilizers
of $\g q$ and of $G/H$.
This will allow us to state 
in Sections \ref{secredsemr}, \ref{secexacom} and \ref{secexarea},
a few direct consequences of what we have proven so far. 

In Section \ref{seccon}, we give two delicate examples 
of real reductive homoge\-neous spaces that one shall have in mind when looking for a more precise converse of Theorem \ref{thghtemp} (1).

\subsection{Generic stabilizer in $\g g/\g h$ and in $G/H$}
\label{secgenstaqgh}

Let $G$ be a semisimple algebraic Lie group,
 $H$ a reductive subgroup, $\g g$, $\g h$ their Lie algebras and 
$\g q=\g g/\g h$.

For $x$ in $G/H$, we denote by $\g h_x$ the stabilizer of $x$ in $\g h$.
As in Definitions \ref{defrgs} and \ref{defags}, we say that
$G/H$ has RGS (resp. $\rm{AGS}$, $\rm{AmGS}$) in $\g h$ if the set 
$\{ x\in G/H\mid \g h_x
\;\mbox{is reductive (resp. abelian reductive, amenable reductive)}\}$
 is dense in $G/H$.

The following Lemmas \ref{lemgenstaqgh1} and \ref{lemgenstaqgh2} relate
the generic stabilizers of $G/H$ in $\g h$  
and the generic stabilizers of $\g q$ in $\g h$. 
The first lemma should be compared with Lemma \ref{lemredstacom}.
 
\begin{lemma}
\label{lemgenstaqgh1}
Let $G$ be a real semisimple algebraic Lie group,
 and $H$ a reductive algebraic subgroup. 
Then $G/H$ has RGS in $\g h$. More precisely,
there exists finitely many reductive Lie subalgebras
$\g m_1,\ldots, \g m_r$ of $\g h$
such that the set 
of $x$ in $G/H$ whose stabilizer $\g h_x$ in $\g h$ is conjugate 
 by the adjoint group $H$
 of ${\mathfrak{h}}$
to one of the $\g m_i$
contains a non-empty Zariski open subset of $G/H$.
\end{lemma} 

The Lie algebras $\g m_i$ which cannot be removed from this list
will be called the {\it{generic stabilizers}} of $G/H$.
They  are well defined only up to conjugacy and permutation.

\begin{lemma}
\label{lemgenstaqgh2}
Let $G$ be a real semisimple algebraic Lie group,
$H$ be a reductive algebraic subgroup, and $\g q=\g g/\g h$. 
Then $G/H$ and $\g q$ have the same set of generic stabilizers in $\g h$. 
In particular, one has the equivalences:
\\
{\rm{(1)}}\enspace
$G/H$ has $\rm{AGS}$ in $\g h$ $\Longleftrightarrow$ 
$\g q$ has $\rm{AGS}$ in $\g h$.
\\
{\rm{(2)}}\enspace
$G/H$ has AmGS in $\g h$ $\Longleftrightarrow$ 
$\g q$ has AmGS in $\g h$.
\end{lemma} 

\begin{proof}[Proof of Lemmas \ref{lemgenstaqgh1} and \ref{lemgenstaqgh2}]
By Chevalley's theorem, there exists a finite-dimensional 
representation $V$ of $G$ and a point $v$ in $V$ whose stabilizer in $G$ is $H$.
The tangent space at $v$ to the $G$-orbit $G\, v$ is isomorphic 
to $\g g/\g h$ as a representation of $H$. 
Since $H$ is reductive, there exists an $H$-invariant decomposition
$V=T_v(G/H)\oplus W$. 
In particular, there is an $H$-equivariant projection 
$\pi \colon V\rightarrow \g g/\g h$, which induces an $H$-equivariant dominant map
$\pi \colon G/H\rightarrow \g g/\g h$. 
In particular, there exists an open Zariski dense subset 
$U$ of $G/H$, such that, for all $x$ in $U$, 
$x$ and $\pi(x)$ have same stabilizer in $\g h$.
Note that $\ol{\pi(U)}$ contains a neighborhood of $0$,
and that $\g h_{v}=\g h_{tv}$, for all $v$ in $V$, $t$
 in $\m R \setminus \{0\}$.  
Our claims follow.
\end{proof}

\subsection{Reductive and semisimple subgroups}
\label{secredsemr}

The following proposition reduces our classification to the case of a 
semi\-simple Lie subgroup $H$ without compact factor.

\begin{pro} 
\label{progh1h2}
Let $G$ be a real semisimple algebraic Lie group,
and
$H_1\supset H_2$ two unimodular subgroups.\\
{\rm{(1)}} If $L^2(G/H_1)$ is 
tempered
then $L^2(G/H_2)$ is 
tempered.\\
{\rm{(2)}} The converse is true when $H_2$ is normal in $H_1$ and $H_1/H_2$ is amenable
(for instance, finite, or compact, or abelian).
\end{pro}


The following proposition follows from Proposition \ref{prorhrqreacom}, 
 and reduces our classification to the case of a simple 
Lie group $G$.

\begin{pro} 
\label{progg1grh}
Let $G$ be a real semisimple algebraic Lie group, 
$H$ a real reductive algebraic subgroup of $G$.  
Let $G_i$ $(1 \le i \le r)$ be simple factors of $G$, 
 and we set $H_i:=H \cap G_i$.  
The representation of $G$ in $L^2(G/H)$ is tempered
if and only if,  for all $i\leq r$, the representation of $G_i$ in  
$L^2(G_i/H_i)$ is tempered.
\end{pro}

The following proposition is an easy corollary of our criterion \eqref{eqnl2ghrhrq}. 

\begin{pro} 
\label{porgh}
Let $G$ be a real semisimple algebraic Lie group,
and
$H$ a real reductive algebraic subgroup.\\
{\rm{(1)}} If the represen\-tation of $G_{\mathbb{C}}$ in
$L^2(G_{\mathbb{C}}/H_{\mathbb{C}})$ is 
tempered,
then the represen\-tation of $G$ in $L^2(G/H)$ is 
tempered.\\
{\rm{(2)}} The converse is true when $H$ is a split group.
\end{pro}

\subsection{Examples of complex homogeneous spaces}
\label{secexacom}

In this section
 we give a few examples of complex homogeneous spaces $G/H$ 
 where $G$ and $H$ are {\it{complex}} Lie groups.
We recall 
 that Theorem \ref{thrhrqc} together
 with the criterion \eqref{eqnl2ghrhrq}
 implies the following:

\begin{cor}
\label{exagxg2}
Suppose $G$ is a complex semisimple algebraic group
and $H$ a complex reductive subgroup.
Then the representation of $G$ in 
$L^2(G/H)$ is tempered if and only if the set of points
 in $\g g/\g q$ 
with abelian stabilizer in $\mathfrak{h}$ is dense.
\end{cor}

\begin{example}
\label{exagslsoc}
$L^2({\rm SL}(n,\mathbb{C})/{\rm SO}(n,\mathbb{C}))$ is always tempered.\\
$L^2({\rm SL}(2m,\mathbb{C})/{\rm Sp}(m,\mathbb{C}))$ is never tempered.\\
$L^2({\rm SO}(7,\mathbb{C})/G_2)$ is not tempered.
\end{example}

The first two cases above are symmetric spaces, 
 see also Example \ref{exag1ck1c}.  
The next example is a consequence of Proposition \ref{proslslsl}.  

\begin{example}
\label{exaslnc}
Let $n=n_1+\cdots + n_r$ with $n_1\geq\cdots\geq n_r\geq 1$, $r\geq 2$.
\\
$L^2({\rm SL}(n,\mathbb{C})/\prod {\rm SL}(n_i,\mathbb{C}))$ is tempered iff
$2n_1\leq n+1$.\\
$L^2({\rm SO}(n,\mathbb{C})/\prod {\rm SO}(n_i,\mathbb{C}))$ is tempered iff
$2n_1\leq n+2$.\\
$L^2({\rm Sp}(n,\mathbb{C})/\prod {\rm Sp}(n_i,\mathbb{C}))$ is tempered iff
$r\geq 3$ and $2n_1\leq n$.
\end{example}

\subsection{Examples of real homogeneous spaces}
\label{secexarea}

Here are a few examples of application of our criterion \eqref{eqnl2ghrhrq}.

\begin{example}
\label{exagxg}
Let $G_1$ be a real semisimple algebraic Lie group and 
$K_1$ a maximal compact subgroup.\\
{\rm{(1)}}$L^2(G_1\times G_1/\Delta(G_1))$ is always tempered.\\
{\rm{(2)}}
$L^2(G_{1,\mathbb{C}}/G_1)$ is always tempered.
\\
The first statement is obvious from the definition of temperedness, 
and alternatively follows immediately from \eqref{eqnl2ghrhrq}
and Proposition \ref{prorhrqreacom}.  
The second statement follows from the first one
 as a special case
 of the example below.  
\end{example}

\begin{example}
\label{exafj} 
Let $G/H$ be a symmetric space 
{\it{i.e.}} $G$ is a real semisimple algebraic Lie group
 and $H$ is the set of fixed points
of an involution of $G$.
Write $\mathfrak{g}=\mathfrak{h}\oplus \mathfrak{q}$ for the $H$-invariant decomposition of $\mathfrak{g}$.
Let $G^c$ be a semisimple algebraic Lie group with Lie algebra 
$\mathfrak{g}^c=\mathfrak{h}\oplus \sqrt{-1}\mathfrak{q}$,
so that the  $\g h$-modules $\g g/\g h$ and $\g g^c/\g h$
are isomorphic. Therefore, \\
$L^2(G/H)$ is tempered iff $L^2(G^c/H)$ is tempered.  
\end{example}

\begin{example}
\label{exagslopq}
$L^2({\rm SL}(p+q,\mathbb{R})/{\rm SO}(p,q))$ is always tempered.\\
$L^2({\rm SL}(2m,\mathbb{R})/{\rm Sp}(m,\mathbb{R}))$ is never tempered.\\
$L^2({\rm SL}(m+n,\mathbb{R})/{\rm SL}(m,\mathbb{R})\times {\rm SL}(n,\mathbb{R}))$ is tempered iff $|m-n|\leq 1$.
\end{example} 

\begin{example}
\label{ex:soprod}
Let
$p_1 + \dots + p_r\le p$ and  $q_1 + \dots + q_r\le q$.\\
$L^2({\rm SO}(p,q)/\prod {\rm SO}(p_i,q_i))$ is tempered iff
$\displaystyle 
2\max_{p_iq_i\neq 0} (p_i+q_i)\leq p+q+2$.
\end{example}

The homogeneous spaces in Examples~\ref{exaslnc} and \ref{ex:soprod}
are not symmetric spaces when $r \ge 3$.
In most cases, 
 they are not even real spherical
 (\cite{xKOmf}), 
 too.

\subsection{About the converse of Theorem \ref{thghtemp}}
\label{seccon}

Even when $H$ has no compact factors and $G/H$ is a reductive symmetric space, 
the converses of the implications  in Theorem \ref{thghtemp}
are not always true. 
Here are two examples that
follow from Theorem \ref{thghtempr}.

$(1)$  
Counterexample of the converse of Theorem \ref{thghtemp} (1).

$
L^2({\rm Sp}(p_1+p_2,q_1+q_2)/{\rm Sp}(p_1,q_1)\times {\rm Sp}(p_2,q_2) )
$
is not tempered when 
$p_1\geq 1$, $q_1\geq 1$  and
$
p_1+q_1= p_2+q_2 +1
$,
even though the  set of points in $G/H$ with amenable stabilizer in $H$ is dense.

$(2)$ 
Counterexample of the converse of Theorem \ref{thghtemp} (2).

$
L^2({\rm SL}(2m-1,\mathbb{H})/S({\rm GL}(m,\mathbb{H})\times {\rm GL}(m-1,\mathbb{H})) )$ 
is tempered when
$m\geq 2$
even though the  
set of points in $G/H$ with abelian stabilizer in $\mathfrak{h}$ is not dense.



{\small
\noindent
Y. \textsc{Benoist}\newline
CNRS-Universit\'e Paris-Saclay, Orsay, France\newline
\texttt{yves.benoist@u-psud.fr}

\medskip
\noindent
T. \textsc{Kobayashi} :
Graduate School of Mathematical Sciences and 
Kavli IPMU (WPI), The University of Tokyo, Komaba,  Japan\newline
\texttt{toshi@ms.u-tokyo.ac.jp}

\end{document}